\documentclass{amsart}
\usepackage{amstext,amssymb,hyperref,color,stmaryrd,tikz,graphicx, epstopdf}

\newtheorem{theorem}{Theorem}[section]
\newtheorem{lemma}[theorem]{Lemma}

\theoremstyle{definition}

\theoremstyle{remark}
\newtheorem{remark}[theorem]{Remark}

\theoremstyle{corollary}
\newtheorem{corollary}[theorem]{Corollary}
\numberwithin{equation}{section}

\newcommand\mI{\mathcal{I}}
\newcommand\mT{\mathcal{T}}
\newcommand\mE{\mathcal{E}}
\newcommand\mM{\mathcal{M}}
\newcommand\mF{\mathcal{F}}

\newcommand\mR{\mathcal{R}}

\newcommand\Hgamma{H^1_{\Gamma_0}(\Omega)}
\newcommand\osc{\mathrm{osc}}
\newcommand\VT{\mathbb{V}_{\mT}}
\newcommand\ci{\textnormal{i}}
\newcommand\half{\frac{1}{2}}
\newcommand\halfa{\frac{\alpha}{2}}

\newcommand\tildeeta{\widetilde{\eta}}
\newcommand\tilderho{\widetilde{\rho}}
\newcommand\tildeM{\widetilde{\mathcal{M}}}
\newcommand\tildeR{\widetilde{R}}

\newcommand\tildeosc{\widetilde{\osc}}

\newcommand\condition{k^3 h_0^{1+\alpha}}
\newcommand\conditionT{k^3 h_\mT^{1+\alpha}}
\newcommand\V[1]{\mathbb{V}_{#1}}
\newcommand\T[1]{\mathcal{T}_{#1}}
\newcommand\M[1]{\mathcal{M}_{#1}}
\newcommand\energy[1]{\interleave #1 \interleave}
\newcommand\twonorm[1]{\|#1 \|}
\newcommand\intbound[1]{\langle #1 \rangle}
\newcommand\bdnorm[1]{\| #1 \|_{\Gamma_1}}
\newcommand\tildeu[1]{\widetilde{u}_{#1}}
\newcommand\jump[1]{\llbracket #1 \rrbracket}
\newcommand\hone[1]{\twonorm{#1}_1}
\newcommand\htwo[1]{\twonorm{#1}_2}

\newcounter{mycnt}

\newcommand{\norm}[1]{\left\Vert#1\right\Vert}

\newcommand{\abs}[1]{\left\vert#1\right\vert}
\newcommand{\pd}[1]{\left\langle #1\right\rangle}

\newcommand{\set}[1]{\left\{#1\right\}}

\newcommand{\db}{\displaybreak[0]}

\newcommand{\ls}{\lesssim}
\newcommand{\al}{\alpha}

\newcommand{\de}{\delta}
\newcommand{\De}{\Delta}

\newcommand{\ga}{\gamma}
\newcommand{\Ga}{\Gamma}

\newcommand{\La}{\Lambda}
\newcommand{\na}{\nabla}
\newcommand{\om}{\omega}
\newcommand{\Om}{\Omega}
\newcommand{\pa}{\partial}

\newcommand{\ta}{\theta}

\newcommand{\Llb}{\La_{\rm lb}}
\newcommand{\Lloc}{\La_{\rm loc}}
\newcommand{\tLub}{\tilde\La_{\rm ub}}
\newcommand{\tLlb}{\tilde\La_{\rm lb}}
\newcommand{\Lcvg}{\La_{\rm cvg}}

\newcommand{\Clb}{C_{\rm lb}}
\newcommand{\Ccvg}{C_{\rm cvg}}
\newcommand{\Copt}{C_{\rm opt}}
\newcommand{\Cmark}{C_{\rm mark}}

\DeclareMathOperator{\re}{{Re}}
\DeclareMathOperator{\im}{{Im}}

\newcommand{\eq}[1]{\begin{align}#1\end{align}}
\newcommand{\eqn}[1]{\begin{align*}#1\end{align*}}

\newcommand{\qtq}[1]{\quad\text{#1}\quad}

\begin{document}

\title[AFEM for Helmholtz Equation]{Adaptive FEM for  Helmholtz Equation\\ with Large Wave Number}
\markboth{ S. Duan and H. Wu}{AFEM for Helmholtz Equation}

\author[S. Duan]{Songyao Duan}
\address{Department of Mathematics, Nanjing University, Jiangsu, 210093, P.R. China. }
\curraddr{}
\email{dsy@smail.nju.edu.cn}
\thanks{This work was partially supported by the NSF of China under grant 11525103.}

\author[H. Wu]{Haijun Wu}
\address{Department of Mathematics, Nanjing University, Jiangsu, 210093, P.R. China. }
\curraddr{}
\email{hjw@nju.edu.cn}
\thanks{}


\date{}

\dedicatory{}


\begin{abstract}
A posteriori upper and lower bounds are derived for the linear finite element method (FEM) for the Helmholtz equation with large wave number. It is proved rigorously that the standard residual type error estimator seriously underestimates the true error of the FE solution for the mesh size $h$ in the preasymptotic regime, which is first observed by Babu\v{s}ka,~et~al. for an one dimensional problem. By establishing an equivalence relationship between the error estimators for the FE solution and the corresponding elliptic projection of the exact solution, an adaptive algorithm is proposed and its convergence and quasi-optimality are proved under  condition that $k^3h_0^{1+\al}$ is sufficiently small, where $h_0$ is the initial mesh size and $\frac12<\al\le 1$ is a regularity constant depending on the maximum reentrant angle of the domain. Numerical tests are given to verify the theoretical findings and to show that the adaptive continuous interior penalty finite element method (CIP-FEM) with appropriately selected penalty parameters can greatly reduce the pollution error and hence the residual type error estimator for this CIP-FEM is reliable and efficient even in the preasymptotic regime.
\end{abstract}

\maketitle

\section{Introduction}
Let $\Omega_0 \subset\Omega_1\subset \mathbb{R}^2$ be two  convex polygons, we consider the following Helmholtz equation in the domain $\Omega:= \Omega_1 \backslash \bar \Omega_0$ with impedance boundary condition on $\Gamma_1:=\partial \Omega_1$ and homogeneous Dirichlet boundary condition on $\Gamma_0:=\partial \Omega_0$:
\begin{equation}\label{eq:helm}
\begin{cases}
-\De u - k^2 u=f  \quad &\text{in}\   \Omega, \\
\;\;\;  \dfrac{\partial u}{\partial n} - \ci k u= g \quad &\text{on} \  \Gamma_1,\\
\,\ \ \phantom{----} u=0 \quad &\text{on} \ \Gamma_0,
\end{cases}
\end{equation}
where  $\ci=\sqrt{-1}$ denotes the imaginary unit and $n$
denotes the unit outward normal
to $\pa\Om$. The above Helmholtz problem can be used for modeling the
acoustic scattering problem (with time dependence $e^{\ci\om t}$) and $k$ is known as the wave number. The impedance boundary condition is known as the
lowest order approximation of the Sommerfeld radiation condition and the homogeneous Dirichlet boundary condition is known as the sound soft boundary condition (cf. \cite{em79}).
\begin{figure}
\centering
\begin{tikzpicture}
\draw (-2,-2) -- (2,-2) -- (2,2) -- (-2,2) -- (-2, -2);
\node (gamma1) at (-1.7,-1.5){$\Gamma_1$} ;
\draw (-1,0) -- (-0.5,{-0.5*sqrt(3)}) -- (0.5, {-0.5*sqrt(3)}) -- (1,0) -- (-0.5, {0.5*sqrt(3)}) -- (-1,0);
\node (gamma0) at (0,  {-0.5*sqrt(3)+0.15} ) {$\Gamma_0$};
\end{tikzpicture}
\end{figure}

When applied to the Helmholtz problems with large wave number, due to the highly indefiniteness of the problems, the finite element method (FEM) usually possesses some special properties that are different from those when applied to the definite elliptic problems. For example, the constant in the C\'ea Lemma increases as the wave number $k$ increases, which is known as the pollution effect \cite{ib95a,ihlenburg98,bs00,Du2015Preasymptotic,wu2014,liwu2019,zw2013}. In particular, \cite{wu2014,liwu2019} proved preasymptotic a priori error estimate $\norm{u-u_h}_{H^1}=O(kh+k^3h^2)$ for the linear FEM for the Helmholtz equation on convex polyhedra domain under the condition that $k^3h^2$ (or $kh$ for 1D case) is sufficiently small, where $h$ is the mesh size. The first term $O(kh)$ is the interpolation error (or the error of the best approximation), which can be reduced to a given tolerance by putting enough but fixed number of points per wavelength, while the second term $O(k^3h^2)$ can not be reduced in the same way and is called the pollution error.   Another example is that the standard a posteriori error estimator of residual type seriously underestimates the true error of the finite element solution for the mesh size $h$ in the preasymptotic regime, which was observed by Babu\v{s}ka,~et~al.~\cite{Babu1997A,Babu1997B} for one dimensional (1D) problems but no rigorous proof was given there. This phenomenon was also observed in \cite{ib01}.

Ever since the pioneer work of Babu\v{s}ka and Rheinboldt \cite{br78e}, the adaptive finite element method (AFEM) based on the a posteriori error estimates has become a central scheme for numerical simulations of partial differential equations. For definite elliptic problems, the theory of AFEM  has been well-developed. We refer to \cite{dorfler,dataoscillation,convergenceAFEM,AFEMforGeneral,Stevenson,quasi-optimal} for results on a posteriori error estimates,  convergence, and quasi-optimality of the AFEM. While for the Helmholtz equation with large wave number, there are only relative few results in the literature on AFEM and those on convergence and quasi-optimality hold only for $h$ in the asymptotic regime, where the behaviour and the analysis are much like those of the AFEM for definite elliptic problems. D\"{o}rfler and Sauter~\cite{ds13} derived a residual-type a posteriori error estimate for the $hp$-FEM which is reliable and efficient.  Chaumont-Frelet, Ern, and Vohral\'{\i}k \cite{cev21} proposed a $p$-robust a posteriori estimate based on an equilibrated flux.   Let $\om_{\max}$ be the maximum reentrant angle of the domain $\Om$ and let $\al:=\frac{\pi}{\om_{\max}}\in (\frac12, 1)$.  Zhou and Zhu \cite{ACIPforhelmholtz} proved robustness of the residual-type a posteriori error estimator and the convergence of the AFEM (as well as the adaptive CIP-FEM) under the condition that $k^{1+s}h_0^s$ is sufficiently small, where $\frac12<s<\al$ and $h_0$ is the mesh size of the initial mesh.  Du, Wu, and Zhang \cite{dwz20} considered a recovery-type a posteriori error estimator on quasi-uniform meshes satisfying some approximate parallelogram condition  and proved that it does underestimate the error of the finite element solution. For a posteriori analyses of other methods, we refer to \cite{AIPDGforhelmholtz,sz15,kmw15}. We would like to mention the work of Chaumont-Frelet and Nicaise \cite{Highfrequency_behaviour}, in which  a preasymptotic a priori error estimate for the linear FEM for the Helmholtz equation \eqref{eq:helm} with $g=0$ was proved under the condition that $k^3h^{1+\al}$ is sufficiently small.
 To the best of the author's knowledge, the theory of  the AFEM based on the residual-type a posteriori error estimate in the preasymptotic regime is far from mature. For example, can we prove the observation of Babu\v{s}ka,~et~al.~\cite{Babu1997A}? Does the convergence  or quasi-optimality  hold if the initial mesh size is as big as in the preasymptotic regime?

The purpose of this paper is to study the properties of AFEM based on the residual-type a posteriori error estimate in the preasymptotic regime.  First, under the condition that $k^3h^{1+\al}$ is sufficiently small, it is proved that the error estimator of the FE solution is a robust estimate of the error the elliptic projection of the exact solution and hence seriously underestimates the error of the FE solution in the preasymptotic regime. Secondly, an AFEM is proposed and then its convergence and quasi-optimality are established under the initial mesh size condition that $k^3h_0^{1+\al}$ is sufficiently small. Finally, numerical examples are provided to verify the theoretical findings. The key idea of the analysis is to use the elliptic projection of the exact solution to establish a bridge to  the theory of AFEM for definite elliptic problems.

The rest of the paper is organized as follows. In \S~\ref{s1}, the regularity decomposition of the exact solution, the Scott-Zhang interpolation, the linear FEM, and the elliptic projection are introduced. The  a priori error estimates for the linear FEM and the elliptic projection are also provided. In \S~\ref{s2}, a posteriori upper and lower bounds are derived for the errors of the FE solution and the elliptic projection. An equivalent relationship between error estimators for the FE solution and the elliptic projection is established. Based on this equivalent relationship, an adaptive finite element algorithm is proposed, which is slightly different from the standard one, and its convergence is proved in \S~\ref{s3}. The quasi-optimality of the AFEM is proved in \S~\ref{s4}. Numerical tests are given in \S~\ref{s:num}.

For the simplicity of notation, we shall frequently use
$C$ for a generic positive constant in most of the subsequent estimates,
which is independent of the mesh size $h$, the source term $f$, and the exact solution $u$. We will also often write  $A\lesssim B$ and $B\gtrsim A$ for the inequalities $A\leq C B$ and $B\geq CA$ respectively. $A\eqsim B$ is used for an equivalent statement when both $A\lesssim B$ and $B\lesssim A$ hold.  All functions in this paper are complex-valued. The space, norm and inner product notation used in this paper are all standard. Their definitions can be found in \cite{bs08,ciarlet78}. In particular, $(\cdot,\cdot)_Q$ and $\pd{ \cdot,\cdot}_\Sigma$
for $\Sigma\subset \pa Q$ denote the complex $L^2$-inner product
on $L^2(Q)$ and $L^2(\Sigma)$
spaces, respectively.  For simplicity, denote by $(\cdot,\cdot):=(\cdot,\cdot)_\Om$, $\norm{\cdot}_j:=\norm{\cdot}_{H^j(\Om)}$, $\abs{\cdot}_j:=\abs{\cdot}_{H^j(\Om)}$,  $\|\cdot\|:=\|\cdot\|_0=\|\cdot\|_{L^2(\Om)}$,  $\twonorm{\cdot}_{\Gamma_1}:=\twonorm{\cdot}_{L^2(\Gamma_1)}$ and by $\twonorm{\cdot}_{\half, \Gamma_1}:=\twonorm{\cdot}_{H^\half (\Gamma_1)}$. Since we are considering high-frequency problems, we assume that $k\gg 1$. Further, let us assume that both $\Om_0$ and $\Om_1$ are star-shaped with respect to a point in $\Om_0$.

\section{FEM and elliptic projections}\label{s1}
In this section, we recall the  wave-number-explicit stability estimate and the regularity decomposition for the solution to the model problem \eqref{eq:helm} and introduce its finite element approximation and elliptic projection.

\subsection{Regularity decomposition}
Let $\Hgamma:=\{v\in H^1(\Omega): v|_{\Gamma_0}=0 \}$ and $a(\cdot , \cdot)$ be
 the sesquilinear form  on $\Hgamma \times \Hgamma$ defined by
\begin{align}
a(u,v):=&(\nabla u, \nabla v) - k^2 (u,v)-\ci k\intbound{u, v}_{\Gamma_1}. \label{form a}
\end{align}
The variational formulation of \eqref{eq:helm} reads as:
Find $u \in \Hgamma$, such that
\begin{equation}\label{eq:weakhelm}
 a(u,v)=(f,v) + \intbound{g, v}_{\Gamma_1} \quad \forall v \in \Hgamma.
\end{equation}

To state the regularity decomposition, we first introduce some notation.  We denote by $\chi_s (r)$  some  $C^\infty$ cutoff function that equals $1$ if $0\le r\le\frac{s}{3}$ and $0$ if $r\ge\frac{2s}{3}$. Let $\{x_j\}_{j=1}^{J}$ the set of corners on $\Gamma_0$.  To each corner $x_j$ we associate a local polar coordinates system $(r_j, \theta_j)$ centered at $x_j$ and with the polar axis taken along the edge before $x_j$ in counterclockwise order. We denote by $\omega_j$ the angle of the corner at $x_j$ and set $\alpha_j=\frac{\pi}{\omega_j}$, $\alpha=\min_j \alpha_j$. To characterize the singularities at corners, define $s_j(x):=\chi_L(r_j)r_j^{\alpha}\sin(\alpha_j \theta_j) $  around $x_j,$  where $L\eqsim 1$ is a constant such that the disk $B(x_j, L)$ centered at $x_j$ with radius $L$ does not contain other corners, $j=1, \cdots, J.$

The following lemma gives stability estimates \cite{Melenk:J:1995,cf06,stability_for_helm} of the problem \eqref{eq:weakhelm} and a decomposition of the continuous solution into a regular part and singular parts with explicit dependence on the wave number $k$ (cf. \cite{Highfrequency_behaviour}).
\begin{lemma}[Regularity Estimates]\label{thm:regularity_estimates}
For all $k \gtrsim 1$, $f \in L^2(\Omega)$ and $g\in H^\half(\Gamma_1)$, there exists a unique solution $u\in \Hgamma$ to problem \eqref{eq:weakhelm} satisfying the following stability estimates
\begin{equation}\label{eq:H estimates}
k \twonorm{u}+ \hone{u}\ls \twonorm{f}+ \twonorm{g}_{\Gamma_1} .
\end{equation}
Moreover, there exist a function $u_R \in \Hgamma \cap H^2(\Omega)$ and constants $c_k^j\in \mathbb{C}$, such that the following decomposition holds:
\begin{equation}\label{eq:splitting of helm}
 u=u_R+ \sum_{j=1}^{J} c_k^j s_j,
 \end{equation}
with the estimates
\eq{\label{eq:decomp}\htwo{u_R}  &\lesssim k M(f,g)\qtq{and}
 |c_k^j| \lesssim k^{\alpha_j -\half} (\twonorm{f}+\twonorm{g}_{\Gamma_1}), \  j=1,...,J,}
where $M(f,g)=\twonorm{f}+\twonorm{g}_{\Gamma_1} +\frac{1}{k}\twonorm{g}_{\half, \Gamma_1}.$
\end{lemma}

\begin{proof}
For the proof of \eqref{eq:H estimates} we refer to \cite{Melenk:J:1995,cf06,stability_for_helm}. Next we prove \eqref{eq:splitting of helm}--\eqref{eq:decomp}.
Let $u_g$ be the solution to the following Helmholtz equation on $\Om_1$:
\begin{equation*}
\begin{cases}
-\De u_g - k^2 u_g=0  \quad &\text{in}\   \Omega_1, \\
\;\;\;  \dfrac{\partial u_g}{\partial n} - \ci k u_g= g \quad &\text{on} \  \Gamma_1.
\end{cases}
\end{equation*}
Since $\Om_1$ is convex, there hold the following stability and regularity estimates \cite{Melenk_phd,cf06,stability_for_helm}:
\eq{\label{ugsta}k\norm{u_g}+\norm{u_g}_1\ls \norm{g}_{\Ga_1}\qtq{and} \norm{u_g}_2\le k\norm{g}_{\Gamma_1} +\norm{g}_{\half, \Ga_1}.}
Choose two nested subdomains $\Om_2$ and $\Om_3$ such that $\Om_0\subset\subset\Om_2\subset\subset\Om_3\subset\subset\Om_1$ and let $\chi\in C^2(\Om)$ be a cut-off function  satisfying
\eq{\label{chi} \chi=1 \text{ in } \Om_1\setminus\Om_3,\quad\chi=0\text{ in }\Om_2\setminus\Om_0,\qtq{and} \norm{\chi}_{C^2(\Om)}\ls 1.}
Consider $U=u-\chi u_g|_\Om$ which satisfies
\begin{equation*}
\begin{cases}
-\De U - k^2 U=F:=f+\De\chi u_g+2\na\chi\cdot\na u_g  \quad &\text{in}\   \Omega, \\
\;\;\,\,  \dfrac{\partial U}{\partial n} - \ci k U= 0 \quad &\text{on} \  \Gamma_1,\\
\,\,\ \ \phantom{----} U=0 \quad &\text{on} \ \Gamma_0,
\end{cases}
\end{equation*}
Applying \cite[Theorem 3.2, Theorem 3.6, and Theorem 3.7]{Highfrequency_behaviour}, we have the following decomposition for $U$:
\begin{equation}\label{eq:Udecomp}
 U=U_R+ \sum_{j=1}^{J} c_k^j s_j,
 \end{equation}
with the estimates
\eq{\htwo{U_R}  &\lesssim k \norm{F}\qtq{and}
 |c_k^j| \lesssim k^{\alpha_j -\half}\norm{F}, \  j=1,...,J.}
Moreover, from \eqref{ugsta} and \eqref{chi},  we have
\eq{\norm{F}&\ls\norm{f}+\norm{u_g}_1\ls \norm{f}+\norm{g}_{\Ga_1},\\
\norm{\chi u_g}_2&\ls \norm{u_g}_2\ls k\norm{g}_{\Gamma_1} +\norm{g}_{\half, \Ga_1}.\label{chiug2}}
Then the proof of \eqref{eq:splitting of helm}--\eqref{eq:decomp} follows by letting $u_R=U_R-\chi u_g|_\Om$ and using \eqref{eq:Udecomp}--\eqref{chiug2}. This completes the proof of the lemma.
\end{proof}

\subsection{FEM}
Let $\mT$ be
 a regular and conforming triangulation of $\Om$. For any triangle element $T\in\mT$, let $h_T:=|T|^\frac12$ be the size of $T$. Clearly, $h_T\eqsim{\rm diam}\,T$.   Let $h_\mT:=\max_{T\in\mT}h_T$. Introduce the linear finite element space on $\mT$:
\eq{\label{VT} \VT :=\{ v \in \Hgamma :\  v|_{T} \in P_1(T)  \quad \forall T \in \mT \}, }
where $P_1(T) $ denotes the space of all polynomials on $T$ with total degree $\le 1$. The finite element method for \eqref{eq:weakhelm} reads as:
 Find $u_h \in \VT$ such that
\begin{equation} \label{eq:femhelm}
a(u_h,v_h)=(f,v_h) +\intbound{g, v_h}_{\Gamma_1} \quad \forall v_h \in \VT.
\end{equation}

Introduce the energy norm on $H^1(\Omega)$:
 \[ \energy{v}:= \big(\twonorm{\nabla v}^2+k^2 \twonorm{v}^2\big)^\half.\]

Let $\mI_h$ be the Scott-Zhang interpolation operator \cite{sz90} onto $\VT$. There hold the following error estimates (see \cite[(4.3)]{sz90} and \cite[Lemma~5.1]{Highfrequency_behaviour})
\begin{lemma}\label{lem:SZ} $\mI_h$ satisfies the following approximation properties: For any $T\in\mT$,
\eq{\label{SZ1}\norm{v-\mI_hv}_T+h_T^\frac12\norm{v-\mI_hv}_{\pa T}&\ls h_T\abs{v}_{H^1(\tilde T)}, \quad\forall v\in H^1(\Om),\\
\norm{v-\mI_hv}_T+h_T\norm{v-\mI_hv}_{H^1(T)}&\ls h_T^2\abs{v}_{H^2(\tilde T)},\quad\forall v\in H^2(\Om),\label{SZ2}}
where $\tilde T:=\cup\set{T'\in\mT:\; T'\cap T\neq\emptyset}$. Moreover, for the singular functions $s_j$, we have
\eq{\label{SZ3} \twonorm{s_j-\mI_h s_j} +h_\mT\twonorm{s_j- \mI_h s_j}_{1}\lesssim h_\mT^{1+\alpha_j}, \quad j=1,\cdots, J.
}
\end{lemma}

As a consequence of the above lemma and the decomposition \eqref{eq:splitting of helm}, we have the following error estimate for the Scott-Zhang interpolant of the exact solution $u$  under the condition that $kh_\mT\ls 1$:
\eq{\label{SZ4}
\twonorm{u-\mI_h u}+h_\mT\twonorm{u-\mI_h u}_1\ls \big(kh_\mT^2+k^{-\half}(kh_\mT)^{\alpha}h_\mT\big)M(f,g).
}

By using the modified duality argument \cite{zw2013}, the following preasymptotic a priori error estimate for the linear FEM for \eqref{eq:helm} with $g=0$ was proved in \cite[Theorem 5.5]{Highfrequency_behaviour}. Similarly, we have the following a priori error estimate for the case of general $g$, whose proof is almost the same as that of \cite[Theorem 5.5]{Highfrequency_behaviour} and is omitted.
\begin{lemma}\label{lem:a priori} Let $u$ and $u_h$ be the exact solution and its finite element approximation. Assume that $\conditionT$ is sufficiently small. Then
\eqn{\energy{u-u_h}\ls\big(kh_\mT+k^{-\frac12}(kh_\mT)^\al+k^3h_\mT^2\big)M(f,g) .}
\end{lemma}
\begin{remark}\label{rPriori}
{\rm (i)} The error bound consists of two parts: the interpolation error $O(kh_\mT+k^{-\frac12}(kh_\mT)^\al)$  and the pollution error $O(k^3h_\mT^2)$. The pollution error is of the same order as that for Helmholtz problems without singularities \cite{wu2014,zw2013,Du2015Preasymptotic}. The interpolation error consists of two subparts: one is $O(kh_\mT)$ corresponding to the regular component of $u$, another is $O(k^{-\frac12}(kh_\mT)^\al)$ corresponding to the singular components at reentrant corners of $\Ga_0$, Since $\al\in(\frac12,1)$, $\energy{u-u_h}=O(k^3h_\mT^2)$, i.e., the pollution term dominates in the error bound, if $h_\mT\gtrsim k^{-2}$. The interpolation error corresponding to the regular component dominates in the error bound if $k^{-\frac{3-2\al}{2-2\al}}\ls h_\mT\ls k^{-2}$, and the interpolation error due to singularities dominates in the error bound if $h_\mT\ls k^{-\frac{3-2\al}{2-2\al}}$.

{\rm (ii)} Since the FEM on the quasi-uniform meshes does not converges at full order for $h_\mT\ls k^{-\frac{3-2\al}{2-2\al}}$, it is necessary to use AFEM to resolve the singularities at corners of $\Om_0$, to our knowledge,  whose behaviour in the preasymptotic regime is not well-understood.  Other types of singularities due to source or discontinuous coefficient will be considered in other works.

{\rm (iii)} The analysis of higher order FEM requires finer decomposition of the exact solution. Based on the result for non-singular problems \cite{Melenk2010,Melenk2011}, we conjecture that the regular part $u_R$ can be further decomposed into a non-oscillating component in $H^2$ and an oscillating analytic component, more precisely, $u$ can be decomposed as
\eqn{&u=u_{\mathcal{E}}+u_{\mathcal{A}}+ \sum_{j=1}^{J} c_k^j s_j,\\
\norm{u_{\mathcal{E}}}_2\ls M(f,g)&\qtq{and} \norm{u_{\mathcal{A}}}_j\ls k^{j-1}M(f,g), j=0,1,2,\cdots.}
Once such a decomposition holds, the interpolation from the  $p^{\rm th}$ ($p\ge 2$) order Lagrange finite element space satisfies the following error estimate:
\eqn{\energy{u-\mI_hu}\ls\big(h_\mT+k^ph_\mT^p+k^{-\frac12}(kh_\mT)^\al\big) M(f,g),}
which implies that the singularities become significant  if  $h_\mT\ls k^{-1-\frac{1}{2p-2\al}}$. Obviously, adaption starts to work for higher order FEM on coarser meshes than the linear FEM, and is more necessary. We also leave the analysis of higher order AFEM to a future work.

\end{remark}

\subsection{Elliptic projection}
In this subsection, we introduce the elliptic projection of the exact solution, which will be used as a bridge linking the analysis for the AFEM for the Helmholtz problem to the well-developed theory for the AFEM for the elliptic problems \cite[etc.]{dorfler,convergenceAFEM,quasi-optimal}.

Let the sesquilinear form $b(\cdot , \cdot)$ be the inner product on $\Hgamma \times \Hgamma$:
\begin{align}\label{form b}
b(u,v):=&(\nabla u, \nabla v) + (u, v).
\end{align}
For any $\phi \in \Hgamma$,  its elliptic projections $P_h\phi \in \VT $ is defined by
\begin{equation}\label{eq:elliptic fem}
b(P_h\phi, v_h)=b(\phi, v_h) \quad \forall v_h \in \VT.
\end{equation}
Denote by $\tildeu{h}:=P_hu$ the elliptic projection of the exact solution $u$. Clearly, $\tildeu{h}\in\VT$ can be regarded as the finite element approximation of the elliptic problem which is an equivalent variant of the Helmholtz problem \eqref{eq:helm} but regarding $u$ in the right hand side as known.
\begin{equation} \label{eq:elliptic projection}
\begin{cases}
-\De \tilde u +\tilde u=f+(k^2+1) u  \quad &\text{in}\   \Omega, \\
\quad\quad\;\,  \dfrac{\partial \tilde u}{\partial n} =\ci k u +g \quad &\text{on} \  \Gamma_1,\\
\,\ \phantom{---} \tilde u=0 \quad &\text{on} \ \Gamma_0.
\end{cases}
\end{equation}
Clearly, $\tilde u=u$.

The following lemma gives some useful estimates for elliptic projection.
\begin{lemma}\label{lem:elliptic projection estimate}
Let $\phi \in \Hgamma $, $P_h \phi$ is its elliptic projection on $\V{\mT}$, then we have the following estimates
\begin{align}
\hone{\phi-P_h \phi}&=\inf_{v_h\in \VT}\hone{\phi-v_h},\\
\twonorm{\phi-P_h \phi} &\lesssim h_\mT^\alpha \hone{\phi-P_h \phi} \lesssim  h_\mT^\alpha \hone{\phi},\label{Ph1} \\
\bdnorm{\phi-P_h \phi} &\lesssim h_\mT^\halfa \hone{\phi-P_h \phi}\lesssim h_\mT^\halfa \hone{\phi}.\label{Ph2}
\end{align}
\end{lemma}
\begin{proof}
Since \eqref{Ph2} is a direct consequence of \eqref{Ph1} and the trace inequality, it suffices to prove the first inequality in \eqref{Ph1}. Introduce the adjoint problem
\begin{equation*}
\begin{cases}
-\De w +w=\phi-P_h \phi  \quad &\text{in}\   \Omega, \\
\quad\quad\;\;  \dfrac{\partial w}{\partial n} =0 \quad &\text{on} \  \Gamma_1,\\
\;\ \phantom{---} w=0 \quad &\text{on} \ \Gamma_0.
\end{cases}
\end{equation*}
From \cite{elliptic_non-smooth}, we have the following decomposition similar to \eqref{eq:splitting of helm}:
\eqn{ w=w_R + \sum_{j=1}^J c_j  s_j, \quad \htwo{w_R}+\sum_{j=1}^J |c_j| \lesssim \twonorm{\phi-P_h}.}
The rest of the proof follows from the  the standard Aubin-Nitsche trick (or duality argument) \cite{bs08} and the interpolation error estimates \eqref{SZ2}--\eqref{SZ3}. We omitted the details.
\end{proof}

\section{A posteriori error estimates}\label{s2}
In this section, with the help of the elliptic projection, we derive upper and lower bounds for the FEM and establish some equivalent relationship between the a posteriori error estimators for the finite element solution $u_h$ and those for the elliptic projection $\tildeu{h}$, which are crucial for proving the convergence and quasi-optimality of the AFEM.

Let $\mE$ be the set of all edges of the elements in the triangulation $\mT$ and $\mE^{I}$ be the set of all interior edges in $\mT$.
For any $e\in\mE$, let $\Omega_e:=\cup \{T\in \mT:\; e\subset \partial T\}$.

\subsection{Error estimators} In this subsection, we introduce the residual-type error estimators for $u_h$ and $\tildeu{h}$.

 For any function $v\in H^1(\Om)^d$, we define its jump across a interior edge $e=T\cap T'\in\mE^I$ as follows.
\begin{align*}
\jump{v}:=
v|_{T} \cdot n_{T}+ v|_{T'} \cdot n_{T'},
\end{align*}
where  $n_{T}$ and $n_{T'}$ denote the  unit outward normal on $\partial T$ and $\partial T'$, respectively.

For any $v_h\in \VT$, as usual (cf. \cite{convergenceAFEM,quasi-optimal}), the element residuals, jump residuals, error estimators and oscillations for $v_h$ regarded as an approximation to the Helmholtz problem \eqref{eq:helm} are defined as follows:
\begin{align*}
R_T(v_h):=&(f+k^2 v_h)|_{T}, \quad T \in \mathcal{T},\\
R_e(v_h):=&
\begin{cases}
-\half \jump{\nabla v_h}, &e \subset \Omega,  \\
-\dfrac{\pa v_h}{\pa n} + \ci k v_h+g, & e \subset \Gamma_1,\\
0, & e \subset \Gamma_0,
\end{cases}\quad e\in\mE, \quad R_{\pa T}|_e:=R_e,\quad\forall e\subset\pa T,\\
\eta^2_T(v_h):=&h^2_T \twonorm{ R_T(v_h)}^2_T+ h_T \twonorm{R_{\pa T}(v_h) }^2_{\partial T}, \\
\osc^2_T(v_h):=&h^2_T \twonorm{R_T(v_h)-\overline{R_T(v_h)}}^2_T+  h_T\twonorm{R_{\pa T}(v_h)-\overline{R_{\pa T}(v_h)}}^2_{\partial T}\\
=&h_T^2\twonorm{f-\bar f}_T^2+h_T\twonorm{g-\bar g}_{\partial T \cap \Gamma_1}^2,
\end{align*}
where $\bar{v}|_T$ denotes the $L^2(T)$ projection of $v|_T$ onto $P_1(T)$, $\overline{R_{\pa T}(v_h)}|_e=\overline{R_e(v_h)}$, and $\bar{v}|_e$ denotes the $L^2(e)$ projection of $v|_e$ onto $P_1(e)$. Obviously, $\osc_T^2(v_h)$ is independent of $v_h$.

The element residuals, jump residuals, error estimators and oscillations for $v_h$ regarded as an approximation to the elliptic problem \eqref{eq:elliptic projection} are defined as follows:
\begin{align*}
\tildeR_T(v_h)&:=(f+(k^2+1)u-v_h)|_T, \quad T \in \mathcal{T},\\
\tildeR_e(v_h)&:=
\begin{cases}
-\half \jump{\nabla v_h}, &e \subset \Omega,  \\
-\dfrac{\pa v_h}{\pa n} +\ci k u  + g,  & e \subset \Gamma_1,\\
0, & e \subset \Gamma_0,
\end{cases}\quad e\in\mE,\quad \tildeR_{\pa T}|_e:=\tildeR_e,\quad\forall e\subset\pa T,\\
\tildeeta^2_T(v_h)&:=h^2_T \| \tildeR_T(v_h)\|^2_T+ h_T \| \tildeR_{\pa T}(v_h) \|^2_{\partial T}, \\
\tildeosc^2_T(v_h)&:=h^2_T \| \tildeR_T(v_h)-\overline{\tildeR_T(v_h)} \|^2_T+ h_T \twonorm{\tildeR_{\pa T}(v_h)-\overline{\tildeR_{\pa T}(v_h)}}^2_{\partial T}
\end{align*}

Let $u\in \Hgamma$ be the solution to Helmholtz problem\eqref{eq:weakhelm}, $u_h \in \VT$ be its finite element approximation given in \eqref{eq:femhelm}, and  $\tildeu{h}\in \VT$ be its elliptic projection defined by \eqref{eq:elliptic fem}. The element residuals, jump residuals, error estimators and oscillations for $u_h$ and $\tildeu{h}$ are denoted respectively as follows:
\begin{align}
&R_T:=R_T(u_h),\quad R_e:=R_e(u_h),\quad \eta_T:=\eta_T(u_h),\quad  \osc_T:=\osc_T(u_h),\label{RT}\\
&\tildeR_T:=\tildeR_T(\tildeu{h}),\quad\tildeR_e:=\tildeR_e(\tildeu{h}),\quad \tildeeta_T:=\tildeeta_T(\tildeu{h}),\quad  \tildeosc_T:=\tildeosc_T(\tildeu{h}).\label{tRT}
\end{align}
Note that the quantities on $u_h$ are computable and will be used in the adaptive algorithm, while the quantities on $\tildeu{h}$ are not since they involves the exact solution $u$, which will be used only in the theoretical analysis.

Furthermore, for any $\mM\subset \mT$, denote by:
\eq{\eta_\mM&:= \Big(\sum_{T\in \mM}\eta^2_T \Big)^\half, \quad
     \osc_\mM:=\Big(\sum_{T\in \mM}\osc^2_T\Big)^\half,\label{M1}\\
     \tildeeta_\mM&:= \Big(\sum_{T\in \mM}\tildeeta^2_T \Big)^\frac{1}{2},\quad
\tildeosc_\mM :=\Big(\sum_{T\in \mM}\tildeosc^2_T\Big)^\half.\label{M2}}

\subsection{Upper bounds} In this subsection,
we  derive a posteriori upper bounds for the errors of $u_h$ and $\tildeu{h}$. For simplicity, denote by
\eqn{\rho:= u-u_h,\quad \tilderho:= u-\tildeu{h},}
  and by
\eq{\label{eq:t}
t(k,h):=k^\half h^\halfa +k^{\frac{3}{4}+\halfa}h^\alpha+k^\frac{3}{2} h^\frac{1+\alpha}{2}+k^{\frac{3}{2}+\alpha}h^{2\alpha}+k^3 h^{1+\alpha}.}
 Obviously, $t(k,h_\mT)$ may be arbitrarily small as long as $k^3 h_\mT^{1+\alpha}$ is small enough. Moreover, there hold the following Galerkin orthogonalities:
\eq{\label{GO}
a(\rho,v_h)=0\qtq{and} b(\tilderho,v_h)=0\quad\forall v_h\in \VT.}

\begin{lemma}[Upper bounds] \label{lem:upper bound}
Let $u\in \Hgamma$ and $u_h \in \VT$ be the solutions to \eqref{eq:weakhelm} and \eqref{eq:femhelm}, respectively, and let $\tildeu{h}\in \VT$ be the elliptic projection of $u$. Then we have the following estimates,
\eq{
 \energy{u-u_h} &\lesssim \big(1+k^{\half}(kh_\mT)^{\alpha}+k^2h_\mT\big) \eta_\mT, \label{UB1}\\
  \|u-u_h\| &\lesssim \big(kh_\mT+k^{-\half}(kh_\mT)^{\alpha}\big)\eta_\mT, \label{UB2}\\
 k^\frac12\twonorm{u-u_h}_{\Gamma_1}&\lesssim \big(1+k h_\mT^\half
+k^{\frac14} (kh_\mT)^\halfa\big)\eta_\mT,\label{UB3}\\
 \twonorm{u-\tildeu{h}}_1 &\lesssim (1+k^\half h_\mT^\halfa + k^3 h_\mT^{1+\alpha}+k^{\frac{3}{2}+\alpha} h_\mT^{2\alpha}) \eta_\mT,\label{UB4}\\
\twonorm{u-\tildeu{h}}_1 &\le \tLub\tildeeta_\mT,\label{UB5}  }
where $\tLub$ is a constant independent of $k$ and $h_\mT$.
\end{lemma}
\begin{proof} We first prove \eqref{UB2} by using the duality argument.
Let $w\in \Hgamma$ satisfy
\eqn{ a(v,w)=(v,\rho)\quad \forall v \in \Hgamma.}
Similar to Lemma~\ref{thm:regularity_estimates}, there holds the following decomposition:
\eqn{ w= w_R+\sum_{j=1}^J \tilde c^j_k s_j,}
where $w_R\in \Hgamma \cap H^2(\Omega) $, $ \tilde c^j_k\in \mathbb{C}$, $ s_j(x)=\chi_L(r_j)r_j^{\alpha_j}\sin(\alpha_j\theta_j)$, and it holds that
\begin{align*}
\|w_R\|_2 &\lesssim k \twonorm{\rho}\qtq{and} |\tilde c^j_k| \lesssim  k^{\alpha_j-\half}  \twonorm{\rho}, \quad j=1,...,J.
\end{align*}
Similar to \eqref{SZ4}, we have the following error estimate for the Scott-Zhang interpolation :
\eq{\label{SZw}
\twonorm{w-\mI_h w}_1\ls \big(kh_\mT+k^{-\half}(kh_\mT)^{\alpha}\big)\twonorm{\rho}.}
From \eqref{GO}, \eqref{form a}--\eqref{eq:weakhelm}, and integration by parts, we conclude that
\begin{align}
\twonorm{\rho}^2=&a(\rho, w)=a(\rho, w-\mI_hw) \notag\\
=&a(u, w-\mI_hw)-a(u_h, w-\mI_hw)\notag\\
=&(f, w-\mI_hw)+\intbound{g, w-\mI_hw}_{\Gamma_1} + k^2(u_h,w-\mI_hw)\notag\\
&-\sum_{T\in\mT}\int\nolimits_{\partial T}\frac{\partial u_h}{\partial n}(w-\mI_hw)+\ci k\int\nolimits_{\Gamma_1}u_h(w-\mI_hw)\notag\\
\label{rhoL2}=&\sum_{T\in\mT}\Big((R_T,w-\mI_hw)_T+\sum_{e\subset\pa T}\langle R_e,w-\mI_hw\rangle_e\Big).
\end{align}
Therefore, it follows from \eqref{SZ1} (with $v=w-\mI_hw$) and \eqref{SZw} that
\begin{align*}
\twonorm{\rho}^2 \lesssim & \sum_{T \in \mT} \big(h_T \twonorm{R_T}_T+h_T^\half \twonorm{R_e}_{\partial T}\big)
\twonorm{\nabla(w-\mI_h w)}_{\tilde T} \\
\ls& \eta_\mT\twonorm{\nabla(w-\mI_h w)}\\
               \lesssim&   \eta_T\big(kh_\mT+k^{-\half}(kh_\mT)^{\alpha}\big)\twonorm{\rho},
\end{align*}
which implies that \eqref{UB2} holds.

Next we prove \eqref{UB1}. Similar as above, and using \eqref{form a} and \eqref{GO}, we have
\begin{align*}
\twonorm{\nabla \rho}^2=&\re\big(a(\rho,\rho)\big)+k^2\twonorm{\rho}^2 =\re\big(a(\rho,\rho-\mI_h\rho)\big)+k^2\twonorm{\rho}^2\\
=&\re\sum_{T\in\mT}\Big((R_T,\rho-\mI_h\rho)_T+\sum_{e\subset\pa T}\langle R_e,\rho-\mI_h\rho\rangle_e\Big) +k^2\twonorm{\rho}^2      \\                                            \lesssim&\eta_{\mT}\twonorm{\nabla\rho}+k^2\twonorm{\rho}^2,
\end{align*}
which together with \eqref{UB2} implies that \eqref{UB1} holds.

\eqref{UB3} can be proved as follows
\eqn{
k \bdnorm{\rho}^2=-\im a(\rho, \rho)= -\im a(\rho, \rho-\mI_h\rho) \lesssim  \eta_\mT \twonorm{\nabla\rho} \lesssim  \big(1+k^2h_\mT+k^{\half}(kh_\mT)^{\alpha}\big) \eta_\mT^2.}

Next, we prove \eqref{UB4}. Using \eqref{GO}, \eqref{form a}, and Lemma~\ref{lem:elliptic projection estimate}, we derive that
\begin{align*}
\hone{\tilderho}^2=b(\tilderho,\tilderho)
                                        =&b(\tilderho,\tilderho)=b(\rho,\tilderho)\\
                                         =&a(\rho,\tilderho)+(k^2+1)(\rho,\tilderho)+\ci k \intbound{\rho, \tilderho}\\
                                         =&a(\rho,\tilderho-\mI_h\tilderho)+(k^2+1)(\rho,\tilderho)+\ci k \intbound{\rho, \tilderho}\\
                                         \lesssim&\eta_\mT\twonorm{\nabla\tilderho}+k^2h_\mT^\alpha\twonorm{\rho}\hone{\tilderho}
                                         +kh_\mT^\halfa \bdnorm{\rho}\hone{\tilderho}
\end{align*}
which together with \eqref{UB2}--\eqref{UB3} implies that
\[\hone{\tilderho} \lesssim \big(1+k^\half h_\mT^\halfa + k^3 h_\mT^{1+\alpha}+k^{\frac{3}{2}+\alpha} h_\mT^{2 \alpha}\big)\eta_\mT,\]
that is, \eqref{UB4} holds.

The proof of \eqref{UB5} is standard, we list it below for the reader's convenience:
\begin{align*}
\norm{\tilderho}_1^2=&b(\tilderho,\tilderho)=b(\tilderho, \tilderho-\mI_h\tilderho)\\
=&\sum_{T\in\mT}\Big((\tildeR_T,\tilderho-\mI_h\tilderho)_T+\sum_{e\subset\pa T}\intbound{\tildeR_e,\tilderho-\mI_h\tilderho}_e\Big)\\
\lesssim&\tildeeta_\mT\twonorm{\nabla\tilderho}.
\end{align*}
This completes the proof of the lemma.
\end{proof}

\begin{remark}\label{rUB}
{\rm (i)} The a posteriori error estimate \eqref{UB1} has a structure similar to the a priori error estimate in Lemma~\ref{lem:a priori}, i.e., it consists of the interpolation error $O\big(1+k^{\half}(kh_\mT)^{\alpha}\big) \eta_\mT$ and the pollution error $O(k^2h_\mT) \eta_\mT$, but does not require the mesh condition there.

{\rm (ii)} \eqref{UB4} says that the error of the elliptic projection of $u$ can be bounded by the error estimator for its finite element approximation. In particular,
\eq{\label{BO1} \twonorm{u-\tildeu{h}}_1 &\lesssim  \eta_\mT\qtq{if} \conditionT\ls 1. }

\end{remark}

\subsection{Lower bounds} In this subsection,
we  derive a posteriori lower bounds for the errors of $u_h$ and $\tildeu{h}$. As a byproduct, we prove the observation of Babu\v{s}ka,~et~al.~\cite{Babu1997A}.

\begin{lemma}[Lower bounds] \label{lem:lower bound}
There exist constants $\Clb$, $\Llb$ and $\tLlb$ independent of $k$ and $h_\mT$, such that if  $\conditionT \le \Clb$, then
\begin{align}
 \eta_\mT^2  \le \Llb(\hone{u-\tildeu{h}}^2+\osc_\mT^2),\label{LB1} \\
 \tildeeta_\mT^2 \le \tLlb(\hone{u-\tildeu{h}}^2+\osc_\mT^2).\label{LB2}
\end{align}
Meanwhile we have the local lower bound
\begin{equation}\label{lem:local_lowerbound}
\eta_T \lesssim \twonorm{\nabla (u-u_h)}_{\check{T}} +k\twonorm{u-u_h}_{\check{T}}+\osc_{\check{T}},
\end{equation}
where $\check{T}=T\cup\{T' \in \mT:\; T'\cap T \in \mE \}$.
\end{lemma}

\begin{proof}
Obviously,
\begin{equation}\label{eq:a  Rt Re}
 a(\rho,\phi)=\sum_{T\in\mT}\big((R_T,\phi)_T + \intbound{R_{\pa T},\phi}_{\pa T}\big)
                               \quad \forall \phi \in \Hgamma.
\end{equation}
From \cite[Theorem 3.3]{posteriori}, there exists a bubble function $\varphi_T\ge 0$, which is supported in $T$ and satisfies
\begin{equation}\label{eq:phiT}
\begin{gathered}
\twonorm{\bar{R}_T}_T^2 \eqsim \int_T\varphi_T|\bar{R}_T|^2 \qtq{and}
\twonorm{\varphi_T\bar{R}_T}_T+h_T|\varphi_T\bar{R}_T|_{H^1(T)}\ls \twonorm{\bar{R}_T}_T.
\end{gathered}
\end{equation}
Taking $\phi=\varphi_T\bar{R}_T$ in \eqref{eq:a  Rt Re}, we have $(\tau, \nabla \phi)=0 \ \forall \tau\in (P_0)^d$, it follows from \eqref{eq:phiT}  that
\begin{align*}
\twonorm{\bar{R}_T}_T^2 \lesssim& (\bar{R}_T,\phi)_T=(R_T,\phi)_T+(\bar{R}_T-R_T,\phi)_T\\
 =&a(\rho,\phi)+(\bar{R}_T-R_T,\phi)_T\\
=&(\nabla\tilderho,\nabla\phi)_T-k^2(\rho,\phi)_T+(\bar{R}_T-R_T,\phi)_T\\
\lesssim&\twonorm{\nabla\tilderho}_T\twonorm{\nabla\phi}_T+k^2\twonorm{\rho}_T\twonorm{\phi}_T +\twonorm{\bar{R}_T-R_T}_T\twonorm{\phi}_T\\
\lesssim& h^{-1}_T\twonorm{\nabla\tilderho}_T\twonorm{\bar{R}_T}_T
+k^2\twonorm{\rho}_T\twonorm{\bar{R}_T}_T
+\twonorm{\bar{R}_T-R_T}_T\twonorm{\bar{R}_T}_T,                                                                  \end{align*}
which implies that
\[ h_T\twonorm{\bar{R}_T}_T \lesssim \twonorm{\nabla\tilderho}_T+k^2h_T\twonorm{\rho}_T
                                                                          +h_T\twonorm{\bar{R}_T-R_T}_T.\]
Using triangle inequality, we have
\begin{equation}\label{eq:Rt bound}
h_T\twonorm{R_T}_T \lesssim \twonorm{\nabla\tilderho}_T+k^2h_T\twonorm{\rho}_T
                                                                          +h_T\twonorm{\bar{R}_T-R_T}_T.
\end{equation}

From \cite[Theorem 3.5]{posteriori}, there exist an edge bubble function $\chi_e$, which is supported in $\Om_e$ and satisfies
\begin{equation}\label{eq:chie}
\begin{gathered}
\int_e\chi_e|\bar R_e|^2 \eqsim \twonorm{\bar R_e}_e^2\\
h_T^{-1/2}\twonorm{ \chi_e \bar R_e}_T+h_T^{1/2}|\chi_e \bar R_e|_{H^1(T)}\lesssim \twonorm{\bar R_e}_e,\quad \forall T \subset \Omega_e.
\end{gathered}
\end{equation}

Let $\phi=\sum_{e\in \mE}h_e\chi_e \bar R_e$ in \eqref{eq:a  Rt Re}. From \eqref{GO},  \eqref{form a}, and Lemma~\ref{lem:elliptic projection estimate}, we conclude that
\begin{align*}
a(\rho,\phi)=&a(\rho,\phi-P_h\phi)\\
                     =&b(\rho,\phi-P_h\phi)-(k^2+1)(\rho,\phi-P_h\phi)
                             -\ci k \intbound{\rho,\phi-P_h\phi}_{\Gamma_1}\\
                           =&b(\tilderho,\phi)-(k^2+1)(\rho,\phi-P_h\phi) -\ci k \intbound{\rho,\phi-P_h\phi}_{\Gamma_1}\\
                           \lesssim &\hone{\tilderho}\hone{\phi}+k^2h_\mT^\alpha\twonorm{\rho}\cdot\hone{\phi}
                           +kh_\mT^\halfa\bdnorm{\rho}\hone{\phi}.
\end{align*}
From Lemma \ref{lem:upper bound}, we find that $k^2h_\mT^\alpha\twonorm{\rho}+kh_\mT^\halfa\bdnorm{\rho}\lesssim t(k,h_\mT)\eta_\mT.$
Therefore, from \eqref{eq:a  Rt Re} and \eqref{eq:chie}, we have
\eqn{
\sum_{T\in\mT}\sum_{e\subset\pa T}&\intbound{R_e, \phi}_e
         =a(\rho,\phi)- \sum_{T\in\mT} \Big(R_T, \sum_{e\subset \partial T} h_e\chi_e \bar R_e\Big)_T\\
\lesssim& \hone{\tilderho}\hone{\phi}+k^2h_\mT^\alpha\twonorm{\rho}\cdot\hone{\phi}
                           +kh_\mT^\halfa\bdnorm{\rho}\hone{\phi}\\
                     &+\Big(\sum_{T\in \mT} h_T^2 \twonorm{R_T}_T^2\Big)^\half \Big(\sum_{T\in \mT} h_T \twonorm{\bar R_{\pa T}}_{\pa T}^2\Big)^\half \\
\lesssim&  \big(\hone{\tilderho}+t(k,h_\mT)\eta_\mT\big)\hone{\phi} +\Big(\sum_{T\in\mT}h_T^2\twonorm{R_T}^2_T\Big)^\half
            \Big(\sum_{T\in \mT} h_T \twonorm{ \bar R_{\pa T}}_{\pa T}^2\Big)^\half \\
          \lesssim &\Big(\hone{\tilderho}+t(k,h_\mT)\eta_\mT+\Big(\sum_{T\in\mT}h_T^2\twonorm{R_T}^2_T\Big)^\half\Big)
             \Big(\sum_{T\in \mT} h_T\twonorm{\bar R_{\pa T}}_{\pa T}^2\Big)^\half,
}
where we have used $\hone{\phi}^2\lesssim \sum_{e\in \mE}h_e^2\twonorm{\chi_e \bar R_e}^2_{H^1(\Omega_e)}
\lesssim\sum_{T\in \mT}h_T\twonorm{R_{\pa T}}_{\pa T}^2$ to derive the last inequality.
Thus
\eqn{
\sum_{T\in\mT}h_T\twonorm{\bar R_{\pa T}}_{\pa T}^2
          \eqsim&\sum_{T\in\mT}\sum_{e\subset\pa T}\intbound{\bar R_e, \phi}_e=
         \sum_{T\in\mT}\sum_{e\subset\pa T}\intbound{R_e, \phi}_e+
         \sum_{T\in\mT}\sum_{e\subset\pa T}\intbound{\bar R_e-R_e, \phi}_e\\
          \lesssim &\Big(\hone{\tilderho}+t(k,h_\mT)\eta_\mT+\Big(\sum_{T\in\mT}h_T^2\twonorm{R_T}^2_T\Big)^\half\Big)
             \Big(\sum_{T\in \mT} h_T\twonorm{\bar R_{\pa T}}_{\pa T}^2\Big)^\half \\
             &+\Big(\sum_{T\in \mT} h_T\twonorm{R_{\pa T}-\bar R_{\pa T}}_{\pa T}^2\Big)^\half
               \Big(\sum_{T\in \mT} h_T\twonorm{\bar R_{\pa T}}_{\pa T}^2\Big)^\half.
}
By plugging \eqref{eq:Rt bound} into the above estimate and using \eqref{UB2}, triangle inequality, we conclude that
\begin{equation}\label{eq:Re bound}
\sum_{T\in\mT}h_T\twonorm{R_{\pa T}}_{\partial T}^2\lesssim \hone{\tilderho}^2+t^2(k,h_\mT)\eta^2_\mT+\osc^2_\mT.
\end{equation}
Combining \eqref{eq:Rt bound} with \eqref{eq:Re bound}, we obtain
\[ \eta^2_\mT\lesssim \hone{\tilderho}^2+t^2(k,h_\mT)\eta^2_\mT+\osc^2_\mT.\]
which implies \eqref{LB1} if $\conditionT$ is sufficiently small.

Similarly we can get the local lower bound \eqref{lem:local_lowerbound} by taking $\phi=\varphi_T\bar{R}_T$ and $\phi=h_e\chi_e \bar R_e$ in \eqref{eq:a  Rt Re}, respectively. We omit the details.

Next we prove \eqref{LB2}. From the standard lower bound estimate for elliptic problems (cf. \cite{dataoscillation,convergenceAFEM,quasi-optimal}), we have
\eq{\label{LB3}
 \tildeeta_\mT^2 \ls\hone{u-\tildeu{h}}^2+\tildeosc_\mT^2.
}
It follows from \eqref{RT}--\eqref{tRT} that
\eq{\label{oscosc}|\tildeosc_T-\osc_T|^2
\ls &h^2_T \big\| \tildeR_T(\tildeu{h})-R_T(u_h)-\overline{\tildeR_T(\tildeu{h})-R_T(u_h)} \big\|^2_T\notag\\
&+ h_T \big\|\tildeR_{\pa T}(\tildeu{h})-R_{\pa T}(u_h)-\overline{\tildeR_{\pa T}(\tildeu{h})-R_{\pa T}(u_h)}\big\|^2_{\partial T}\notag\\
\ls &h^2_T \|(k^2+1)\rho\|^2_T+ h_T \twonorm{\ci k\rho}_{\pa T\cap\Ga_1}^2.}
Therefore, from \eqref{M1}--\eqref{M2} and \eqref{UB2}--\eqref{UB3}, we have
\eq{\label{oscosc2}|\tildeosc_\mT - \osc_\mT| &\le \bigg(\sum_{T \in \mT} |\tildeosc_T-\osc_T|^2\bigg)^\half  \notag\\
        &\ls k^2h_\mT \|\rho\|+ kh^\half_\mT \twonorm{\rho}_{\Ga_1} \notag\\
       &\ls t(k,h_\mT)\eta_\mT\ls \eta_\mT,}
Consequently, $\tildeosc_\mT^2\ls \osc_\mT^2+\eta_\mT^2$, which together with \eqref{LB3} and \eqref{LB1} implies that \eqref{LB2} holds. This completes the proof of the lemma.
\end{proof}

\begin{remark}\label{rLB}
{\rm (i)} As a consequence of Lemma~{\rm \ref{lem:upper bound}} (see Remark~{\rm \ref{rUB}(ii)}) and Lemma~\ref{lem:lower bound}, we have proved that under the condition $\conditionT \le \Clb$,  $\eta_\mT\eqsim\norm{u-\tildeu{h}}_1$ up to an oscillation term $\osc_\mT$ which is of higher order if $f$ and $g$ are both piecewise smooth. In other words, the error estimator for the finite element solution $u_h$ for the Helmholtz problem is a good estimate for the error of the elliptic projection of the exact solution $u$, and hence seriously underestimate the error of $u_h$ in the preasymptotic regime. This phenomenon was observed by Babu\v{s}ka,~et~al.~\cite{Babu1997A} for one dimensional problems but no rigorous proof was given there.

{\rm (ii)} The pollution may be reduced by modifying the FEM, e.g., generalized FEM \cite{bips95,Melenk_phd,bs00}, IPDG methods \cite{fw09,hpDG}, CIP-FEM \cite{Du2015Preasymptotic,wu2014,zw2013,ACIPforhelmholtz}, etc. In \S~\ref{s:num}, we will show that by properly selecting the penalty parameter, the adaptive CIP-FEM can reduce greatly the pollution error and hence the residual type a posterior error estimator is good estimate of the true error of the CIP-FE solution even in the preasymptotic regime.

{\rm (iii)} For a posteriori error estimates for the $hp$-FEM in the asymptotic regime, we refer to \cite{ds13}.
\end{remark}
\subsection{Equivalent relationship between two error estimators}
In this subsection, we prove some equivalent relationship between the error estimators for $u_h$ and those for $\tildeu{h}$, which is crucial for the further analysis.

For any submesh $\mM\subset \mT$, let
\eq{\label{tildeM}\tildeM=\{ T \in \mT:\; \exists A \in \mM \text{ such that } A\cap   T \in \mE \}}
 be the set of elements which have common edges with elements in $\mM$. Denote by $\mathcal{F}$ the set of all edges in $\mM$.

The following lemma says that, up to some higher order terms, the error estimator for $u_h$ on $\mM$ is bounded above by the error estimator for $\tildeu{h}$ on a slightly larger submesh $\tildeM$, and vice versa.
\begin{lemma}[The relation between two kinds of estimators]\label{lem:relation between estimators}
There exist $C_1$ and $C_2$ satisfying $C_1C_2 \ge 1$, such that the following estimates hold,
\eq{ \eta_{\mM} &\le C_1\left(  \widetilde{\eta}_{\tildeM} +t(k,h_\mT)\eta_{\mT}\right),\label{etaM1}\\
 \tildeeta_{\mM}& \le C_2\left(  \eta_{\tildeM} +t(k,h_\mT)\eta_{\mT}\right)\label{etaM2}}
\end{lemma}

\begin{proof}
From \eqref{GO}, \eqref{form a}, and \eqref{form b}, we have
\begin{align*}
a(\rho,\phi)=&a(\rho,\phi-P_h\phi)\\
                           =&b(\rho,\phi-P_h\phi)-(k^2+1)(\rho,\phi-P_h\phi)-\ci k\intbound{\rho,\phi-P_h\phi}_{\Gamma_1}\\
                           =&b(\tilderho,\phi)-(k^2+1)(\rho,\phi-P_h\phi)-\ci k\intbound{\rho,\phi-P_h\phi}_{\Gamma_1}.
\end{align*}
Similar to \eqref{eq:a  Rt Re}, we have
\eqn{
 b(\tilderho,\phi)=\sum_{T\in\mT}\big((\tildeR_T,\phi)_T + \intbound{\tildeR_{\pa T},\phi}_{\pa T}\big)
                               \quad \forall \phi \in \Hgamma.
}
Therefore, by combing \eqref{eq:a  Rt Re} with the above two identities, we obtain
\begin{equation}\label{eq:Rt Re tildeRt tildeRe}
\begin{aligned}
\sum_{T\in\mT}\big((R_T,\phi)_T + \intbound{R_{\pa T},\phi}_{\pa T}\big)=&\sum_{T\in\mT}\big((\tildeR_T,\phi)_T + \intbound{\tildeR_{\pa T},\phi}_{\pa T}\big)\\
                                        &-(k^2+1)(\rho,\phi-P_h\phi)-\ci k\intbound{\rho,\phi-P_h\phi}_{\Gamma_1}.
\end{aligned}
\end{equation}

Similar to the proof of Lemma~\ref{lem:lower bound}, let $\phi=\sum_{T\in\mM}h^2_T\varphi_T\bar{R}_T$ in \eqref{eq:Rt Re tildeRt tildeRe}, we have
\[ \sum_{T\in\mM}(R_T,\phi)_T=\sum_{T\in\mM}(\tildeR_T,\phi)_T -(k^2+1)(\rho,\phi-P_h\phi)-\ci k\intbound{\rho,\phi-P_h\phi}_{\Gamma_1} \]
Then
\begin{align*}
\sum_{T\in\mM}h^2_T(\bar{R}_T,\phi_T\bar{R}_T)_T=&\sum_{T\in\mM}h^2_T(R_T,\phi_T\bar{R}_T)_T
                                         +\sum_{T\in\mM}h^2_T(\bar{R}_T-R_T,\phi_T\bar{R}_T)_T\db\\
                                         =&\sum_{T\in\mM}h^2_T(\tildeR_T,\phi_T\bar{R}_T)_T -(k^2+1)(\rho,\phi-P_h\phi)\db\\
                                         &-\ci k\intbound{\rho,\phi-P_h\phi}_{\Gamma_1}                                       +\sum_{T\in\mM}h^2_T(\bar{R}_T-R_T,\phi_T\bar{R}_T)_T\db\\
                                         \lesssim&\sum_{T\in\mM}h^2_T\twonorm{\tildeR_T}_T\twonorm{\bar{R}_T}_T
                                         +k^2h_\mT^\alpha\twonorm{\rho}\hone{\phi}+kh_\mT^\halfa \bdnorm{\rho}\hone{\phi}\\
                                         &+\sum_{T\in\mM}h_T^2\twonorm{R_T-\bar R_T}_T\twonorm{\bar{R}_T}_T,
\end{align*}
where we have used Lemma~\ref{lem:elliptic projection estimate} to derive the last inequality.
Noting from \eqref{eq:phiT} that $\hone{\phi}^2\lesssim \sum_{T\in\mM}h^4_T\twonorm{\phi_T\bar{R}_T}^2_{H^1(T)}\lesssim \sum_{T\in\mM}h^2_T\twonorm{\bar{R}_T}^2_T$  and using Lemma~\ref{lem:upper bound}, we  have
\begin{align*}
\sum_{T\in\mM}h^2_T(\bar{R}_T,\phi_T\bar{R}_T)_T\lesssim& \Big(\sum_{T\in\mM}h^2_T\twonorm{\tildeR_T}^2_T\Big)^\half\Big(\sum_{T\in\mM}h^2_T\twonorm{\bar{R}_T}^2_T\Big)^\half\\
&+t(k,h_\mT)\eta_\mT\Big(\sum_{T\in\mM}h^2_T\twonorm{\bar{R}_T}^2_T\Big)^\half\\
&+\osc_\mM\Big(\sum_{T\in\mM}h^2_T\twonorm{\bar{R}_T}^2_T\Big)^\half,
\end{align*}
thus
\eq{\label{eq:Rt tildeRt} \sum_{T\in\mM}h^2_T\twonorm{\bar{R}_T}_T^2\lesssim& \sum_{T\in\mM}h^2_T\twonorm{\tildeR_T}^2_T+t^2(k,h_\mT)\eta_\mT^2+\osc_\mM^2.}

Let $\phi=\sum_{e\in \mF}h_e\chi_e \bar R_e$ in \eqref{eq:Rt Re tildeRt tildeRe}. From \eqref{eq:chie} and Lemma~\ref{lem:elliptic projection estimate}, we have
\begin{align*}
\sum_{T\in\tildeM}\sum_{\substack{e\subset\pa T\\e\in\mF}}\intbound{R_e, \phi}_e=&
            - \sum_{T\in\tildeM} (R_T, \phi)_T
              +\sum_{T \in \tildeM}\big((\tildeR_T,\phi)_T
               +\intbound{\tildeR_{\pa T},\phi}_{\partial T}\big)\\
           &-(k^2+1)(\rho,\phi-P_h\phi)-\ci k \intbound{\rho, \phi-P_h\phi}_{\Gamma_1}\\
           \lesssim&\Big(\sum_{T\in \tildeM}h^2_T\twonorm{R_T}^2_T\Big)^\half
                        \Big(\sum_{T\in\mM} h_T\twonorm{ \bar R_{\pa T}}_{\partial T}^2\Big)^\half\\
                           &+\Big(\sum_{T\in \tildeM}h^2_T\twonorm{\tildeR_T}^2_T\Big)^{1/2}
                            \Big(\sum_{T\in\mM} h_T\twonorm{\bar R_{\pa T}}_{\partial T}^2\Big)^\half\\
                           &+\Big(\sum_{T\in\mM}h_T\twonorm{\tildeR_{\pa T}}_{\partial T}^2\Big)^{1/2}
                             \Big(\sum_{T\in\mM}h_T\twonorm{ \bar R_{\pa T}}_{\partial T}^2\Big)^{1/2}\\
                           &+k^2h_\mT^\alpha \twonorm{\rho}\hone{\phi}+kh_\mT^\halfa \bdnorm{\rho}\hone{\phi}.
\end{align*}
Noting that \eqref{eq:Rt tildeRt} also hold for $\tildeM$ and
 \[\hone{\phi}^2\lesssim \sum_{e\in \mF}h_e^2\twonorm{\chi_e \bar R _e}_{H^1(\Omega_e)}^2\lesssim \sum_{T\in\mM} h_T \twonorm{ \bar R_{\pa T}}_{\partial T}^2 ,\]
we derive that
\begin{align*}
 \sum_{T\in\mM} h_T \twonorm{\bar R_{\pa T}}_{\partial T}^2
 \eqsim&
    \sum_{T\in\tildeM}\sum_{\substack{e\subset\pa T\\e\in\mF}}\intbound{\bar R_e, h_e\chi_e \bar  R_e}_e \\
 =&
    \sum_{T\in\tildeM}\sum_{\substack{e\subset\pa T\\e\in\mF}}\intbound{R_e, \phi}_e +
    \sum_{T\in\tildeM}\sum_{\substack{e\subset\pa T\\e\in\mF}}\intbound{\bar R_e-R_e, \phi}_e \\
 \ls&
 \sum_{T\in\tildeM}\big(h^2_T\twonorm{R_T}^2_T+h^2_T\twonorm{\tildeR_T}_T^2
                      + h_T\twonorm{\tildeR_{\pa T}}_{\partial T}^2\big)\\
                      &+t^2(k,h_\mT)\eta^2_\mT  + \osc_{\tildeM}^2\\
 \ls&\sum_{T\in\tildeM}\big(h^2_T\twonorm{\tildeR_T}_T^2
                      + h_T\twonorm{\tildeR_{\pa T}}_{\partial T}^2\big) \\
                      &+t^2(k,h_\mT)\eta^2_\mT                     +\osc_{\tildeM}^2.
\end{align*}
Using triangle inequality, we have
\eqn{
\sum_{T\in\mM} h_T \twonorm{R_{\pa T}}_{\partial T}^2\ls&
      \sum_{T\in\tildeM}\big(h^2_T\twonorm{\tildeR_T}_T^2
                      + h_T\twonorm{\tildeR_{\pa T}}_{\partial T}^2\big) \\
                      &+t^2(k,h_\mT)\eta^2_\mT                     +\osc_{\tildeM}^2,}
which together with \eqref{eq:Rt tildeRt} implies that
\eq{\label{etaM3} \eta_{\mM} \lesssim   \tildeeta_{\tildeM} +t(k,h_\mT)\eta_\mT
                                 +\osc_{\tildeM}.}
Similarly, we can obtain,
\eq{\label{etaM4} \tildeeta_{\mM}\lesssim   \eta_{\tildeM} +t(k,h_\mT)\eta_\mT
                                 +\tildeosc_{\tildeM}.}
On the other hand, we have
\eqn{\osc_{\tildeM} \le \eta_{\tildeM}, \quad
\tildeosc_{\tildeM}\le \tildeeta_{\tildeM},}
which together with \eqref{oscosc2} and \eqref{etaM3}--\eqref{etaM4} implies that \eqref{etaM1}--\eqref{etaM2} hold. This completes the proof of the lemma.
\end{proof}

We can get the following corollary by replacing $\mM$ by $\mT$ in lemma \ref{lem:relation between estimators}, which says that the two error estimators are equivalent on $\mT$ if $\conditionT$ is sufficiently small.
\begin{corollary}\label{cor:equivalent of eta}
Suppose $t(k,h_\mT)\le \frac{1}{2C_1}$. Then  $\eta_\mT \eqsim  \widetilde{\eta}_\mT$.
\end{corollary}

The following corollary says that the D\"{o}rfler marking strategy for the finite element solution $u_h$ is equivalent in some sense to that for the elliptic projection  if $\conditionT$ is sufficiently small.
\begin{corollary}\label{cor:dorfler}
Suppose $\theta \in (0, 1)$, $\eta_{\mathcal{M}} \ge \theta \eta_{\mathcal{T}}$, and $ t(k,h_\mT)< \frac{\theta}{C_1}$, then
\[\tildeeta_{\tildeM} \ge \widetilde{\theta}\tildeeta_{\mT}, \qtq{for}   \widetilde{\theta}=\frac{\theta-C_1t(k,h_\mT)}{C_1 C_2(1+t(k,h_\mT))}.\]
\end{corollary}
\begin{proof}
The proof is quite straightforward:
\begin{align*}
\tildeeta_{\tildeM}\ge&\ \frac{\eta_{\mM}}{C_1}-t(k,h_\mT)\eta_{\mT} \ge\ \frac{\theta\eta_{\mT}}{C_1}-t(k,h_\mT)\eta_{\mT}=\ \widetilde{\theta}C_2(1+t(k,h_\mT))\eta_{\mT}\ge\ \widetilde{\theta}\tildeeta_{\mT},
\end{align*}
where we have used \eqref{etaM2} to derive the last inequality. The proof is completed.
\end{proof}

\section{Convergence of AFEM}\label{s3}

In this section, we first state the AFEM algorithm for equation \eqref{eq:weakhelm}, and then prove the convergence of the adaptive algorithm.

\subsection{Adaptive algorithm}

Following the AFEM for elliptic equations, the AFEM for the Helmholtz equation \eqref{eq:weakhelm} can also be described as loops of the following form
\[ \rm SOLVE \longrightarrow ESTIMATE \longrightarrow MARK  \longrightarrow REFINE. \]

We use $n(n \ge 0)$ to denote the iteration number. Let $\mT_0$ be the initial conforming triangulation of $\Omega$ and $\mT_n$ be the conforming triangulation of $\Omega$ in the $n$-th loop. Denote by $\mE_n$ all edges of $\T{n}.$ We abbreviate the finite element space on $\mathcal{T}_n$ as $\V{n}:=\V{\mT_n}$ and the mesh size of $\mathcal{T}_n$ as $h_n:=h_{\mT_n}$.

In the $n$-th loop, we first solve the FEM \eqref{eq:femhelm} on $\mT_n$, and then calculate the error estimators $\eta_T$ for each element $T\in\mT_n$.
Following the estimate step, we use the D\"{o}rfler strategy with parameter $\theta \in (0, 1]$ to mark elements in a subset $\mM_n$ of $\mT_n$ with minimal cardinality such that
\eq{\label{Dorfler} \eta_{\mM_n} \ge \theta \eta_{\mT_n}.}
In the refinement step, we use the newest vertex bisection algorithm  \cite{nvb,the_completion,quasi-optimal} to refine all the elements in $\tildeM_n$ at least $b \ge 1$ times, and then remove the hanging nodes. After refinement, we obtain a new conforming triangulation denoted by $\mT_{n+1}$, which will be used in the next loop.

The basic loop of this adaptive algorithm is given below:
\bigskip

\begin{tabular}{|l|}
 \hline
Given the initial triangulation $\T{0}$ and marking parameter $\theta \in (0,1],$\\
 set $n:=0$ and iterate,\\
1. Solve equation \eqref{eq:femhelm} in $\T{n}$ to obtain $u_n,$\\
2. Calculate the error estimator $\eta_T(u_n)$ on every element $T\in \T{n},$\\
3. Mark  $\mM_n \subset \mT_n$ with minimal cardinality such that
$ \eta_{\mM_n} \ge \theta \eta_{\mT_n},$\\
4. Refine $\tildeM_n$ using the newest vertex bisection algorithm to get $\T{n+1}$,\\
5. Set $n:=n+1$ and then go to step 1.\\
\hline
\end{tabular}
\bigskip

\begin{remark}\label{rAFEM}
 Note that the above AFEM for the Helmholtz equation is almost the same as that for the elliptic equation \cite{dataoscillation,quasi-optimal} except that we have use $\tildeM_n$ as the set of elements to be refined instead of  $\mM_n$.
We do this because we want to use the equivalent relationship between the error estimators for $u_h$ and those for $\tildeu{h}$ given in Lemma~\ref{lem:relation between estimators} and Corollary~\ref{cor:dorfler} to convert the analysis of the above AFEM for the Helmholtz problem to that for the AFEM using elliptic projections for the elliptic problem \eqref{eq:elliptic projection}. More precisely,  besides the pairs of finite solutions and meshes $\set{u_n,\mT_n}$ produced in the AFEM, we also consider the pairs $\set{\tildeu{n},\mT_n} $ where  $\tildeu{n}:=P_hu$ are the elliptic projections of the exact solution $u$ onto the finite element spaces $\V{n}$ on $\mT_n$. From Corollary~\ref{cor:dorfler}, $\set{\tildeu{n},\mT_n}$ can be regarded as pairs generated by some AFEM for the elliptic problem \eqref{eq:elliptic projection}, which uses $\tildeeta_T(\tildeu{n})$ as error estimators. The convergence and quasi-optimality of $\set{\tildeu{n}}$ and $\set{\tildeeta_{\mT_n}}$ may be proved by following the standard analysis for the AFEM for elliptic problems. Then we obtain the convergence and quasi-optimality of $\set{\eta_{\mT_n}}$ by using the equivalent relationship between $\eta_{\mT_n}$ and $\tildeeta_{\mT_n}$. Finally, the convergence and quasi-optimality of $\set{u_{n}}$ are direct consequences of the upper bound estimate \eqref{UB1}.
\end{remark}

\subsection{Convergence of AFEM} In this subsection, we prove the convergence of the AFEM under the condition that $\condition$ is sufficiently small.

Let $u_n \in \V{n}$ and $u_{n+1}\in \V{n+1}$ be the finite element solutions obtained in space $n$-th and $(n+1)$-th steps of AFEM and let $\tildeu{n} \in \V{n}$ and $\tildeu{n+1} \in \V{n+1}$ be the corresponding elliptic projections of exact solution $u$, respectively. Clearly, $\V{n}\subset\V{n+1}$.

By using \eqref{GO}, we can get the orthogonality below directly.
\begin{lemma}[Orthogonality of elliptic projection]\label{lem:orth}
\[\hone{u-v}^2=\hone{u-\tildeu{n}}^2+\hone{\tildeu{n}-v}^2 \quad \forall v\in \V{n}.\]
\end{lemma}

\begin{remark}\label{qor} In \cite{ACIPforhelmholtz}, the authors followed the analysis of definite elliptic problem to prove the convergence of AFEM for the Helmholtz equation, where a quasi-orthogonality for the FEM was used but it required $k^{1+s}h^s$ sufficiently small \cite{ACIPforhelmholtz}, where $\frac12<s<\al$.
\end{remark}

The following lemma gives the estimator reduction property for the elliptic projections, which is proved by following the proof of \cite[Corollary 3.4]{quasi-optimal}. Since the problem \eqref{eq:helm} does not involve variable coefficients but involves the impedance boundary boundary condition, these two proofs are slightly different. So we decide to list the proof here for the reader's convenience.
\begin{lemma}[Estimator reduction] \label{lem:estimator reduction}
 There exists a constant $C_3$ depending only on the shape regularity of $\mathcal{T}_0$, such  that the following inequality holds for any $\delta >0$,
\[ \tildeeta_{\T{n+1}}^2(\tildeu{n+1})   \le   (1+\delta)\big(\tildeeta_{\T{n}}^2(\tildeu{n})                                                        -\lambda\tildeeta_{\M{n}}^2(\tildeu{n})\big)
+ (1+\delta^{-1})C_3 \hone{\tildeu{n}-\tildeu{n+1}}^2 ,\]
where $\lambda:=1-2^{-\frac{b}{2}}.$
\end{lemma}

\begin{proof}
For any $T\in \T{n+1}$ and $v_h,w_h\in \V{n+1}$, using triangle inequality and trace inequality, we derive
\begin{align*}
|\tildeeta_T(v_h)-\tildeeta_T(w_h)|\ls&\big|h^2_T \| \tildeR_T(v_h)-\tildeR_T(w_h)\|^2_T+ h_T \twonorm{\tildeR_{\pa T}(v_h)-\tildeR_{\pa T}(w_h)}^2_{\partial T}\big|^\frac12\\
\lesssim&\big|h^2_T\twonorm{v_h-w_h}^2_T+h_T\twonorm{\jump{\nabla v_h-\nabla w_h}}^2_{\partial T\cap \Omega}+h_T\twonorm{\nabla v_h-\nabla w_h}^2_{\partial T\cap \Gamma_1}\big|^\half\\
                 \lesssim& h_T\twonorm{v_h-w_h}_T+ \twonorm{\nabla(v_h-w_h)}_{\check{T}}\\
                 \lesssim& \twonorm{v_h-w_h}_{H^1(\check{T})}
\end{align*}
 where  $\check{T}=T\cup\{T' \in \T{n+1}:\; T'\cap T \in \mE_{n+1} \}$. The constant hidden in $\lesssim$ depends only on the regularity of $\T{0}$, we use $\tilde C_3$ to denote it. Then Young's inequality with parameter $\delta$ implies,
 \eq{\label{tildeetaT}\tildeeta_T^2(v_h)\le (1+\delta)\tildeeta_T^2(w_h)
  +(1+\delta^{-1})\tilde C_3^2\twonorm{v_h-w_h}_{H^1(\check{T})}^2. }
 For an element $T\in \tildeM_n,$ we define $T_*:=\{T'\in \T{n+1}:\; T'\subset T\},$ then
 \begin{align*}
\sum_{T'\in T_*}\tildeeta_{T'}^2(\tildeu{n})
                                                            =&\sum_{T'\in T_*}\big(h^2_T 2^{-b} \twonorm{\widetilde{R}_T(\tildeu{n})}^2_{T'}
                                                              +h_T 2^{-\frac{b}{2}}\twonorm{ \widetilde{R}_{\pa T'}(\tildeu{n})}^2_{\partial T'}\big)\\
                                                           =&\ 2^{-b} h^2_T \twonorm{ \widetilde{R}_T(\tildeu{n})}^2_T
                                                                 +2^{-\frac{b}{2}} h_T\twonorm{\widetilde{R}_{\pa T}(\tildeu{n})}^2_{\partial T}\\
                                                           \le&\ 2^{-\frac{b}{2}} \tildeeta_T^2(\tildeu{n}),
\end{align*}
therefore
\begin{align*}
\tildeeta_{\T{n+1}}^2(\tildeu{n})=
                          &\sum_{\substack {T\in \tildeM_n\\ T'\in T_*}}\tildeeta_{T'}^2(\tildeu{n})
                                                       +\sum_{\substack {T\in \T{n}\setminus\tildeM_n\\ T'\in T_*}}\tildeeta_{T'}^2(\tildeu{n})\\
                                                       \le&\sum_{T\in \tildeM_n}2^{-\frac{b}{2}}\tildeeta_T^2(\tildeu{n})
                                                        +\sum_{T\in \T{n}\setminus\tildeM_n}\tildeeta_T^2(\tildeu{n})\\
                                                       \le&\ \tildeeta_{\T{n}}^2(\tildeu{n})
                                                        -(1-2^{-\frac{b}{2}}) \tildeeta_{\M{n}}^2(\tildeu{n}).
\end{align*}
Let $\lambda=1-2^{-\frac{b}{2}}$. From \eqref{tildeetaT}, we have
\begin{align*}
\tildeeta_{\T{n+1}}^2(\tildeu{n+1}) \le& \ (1+\delta)\tildeeta_{\T{n+1}}^2(\tildeu{n})+4(1+\delta^{-1}) \tilde C_3^2 \hone{\tildeu{n}-\tildeu{n+1}}^2\\
                        \le&\  (1+\delta)(\tildeeta_{\T{n}}^2(\tildeu{n})                                                        -\lambda\tildeeta_{\M{n}}^2(\tildeu{n}))+(1+\delta^{-1}) 4\tilde C_3^2\hone{\tildeu{n}-\tildeu{n+1}}^2.
\end{align*}
The proof of the lemma is completed by taking $C_3:=4\tilde C_3^2$.
\end{proof}

The following theorem gives the convergence of the AFEM.
\begin{theorem}[Convergence]\label{thm:contraction property}
Let $\theta \in (0,1]$. There exist constants $\Ccvg, \gamma >0$, and $\Lcvg \in (0, 1)$, which depend only on the shape-regularity of $\T{0}, b$, and $\theta$,  such that if $\condition\le \Ccvg$, then
\eqn{\energy{u-u_n} \ls & \Lcvg^n\big(1+k^2h_n\big) \big(kh_0+k^{-\half}(kh_0)^{\alpha}\big) M(f,g). }
\end{theorem}
\begin{proof}  First we note that the elliptic projections $\tildeu{n}$ satisfy the D\"{o}rfler property in Corollary~\ref{cor:dorfler} if $t(k,h_0)\le\frac{\ta}{2C_1}$. Using Lemma \ref{lem:orth} (with $n=n+1$ and $v=\tildeu{n}$) and Lemma \ref{lem:estimator reduction} and following \cite[Theorem 4.1]{quasi-optimal}, we have the following contraction property for the elliptic projections:
\eq{\label{contraction}\hone{u-\tildeu{n+1}}^2+\gamma \tildeeta_{\mT_{n+1}}^2 \le \Lcvg^2(\hone{u-\tildeu{n}}^2+\gamma \tildeeta_{\mT_n}^2).}

Suppose that $\condition$ is so small that $t(k,h_0)\le \frac{\ta}{2C_1}$. It follows from the upper bound \eqref{UB1}, the equivalence given in Corollary~\ref{cor:equivalent of eta},  \eqref{contraction}, and the lower bound \eqref{LB2}  that
\eqn{\energy{u-u_n} &\lesssim \big(1+k^{\half}(kh_n)^{\alpha}+k^2h_n\big) \eta_{\mT_n}\lesssim \big(1+k^2h_\mT\big) \tildeeta_{\mT_n}\\
&\lesssim \big(1+k^2h_n\big) (\hone{u-\tildeu{0}}+ \tildeeta_{0})\Lcvg^n\\
&\lesssim \big(1+k^2h_n\big) (\hone{u-\tildeu{0}}+ \osc_{\mT_0})\Lcvg^n }
Invoking  Lemma~\ref{lem:elliptic projection estimate}, \eqref{SZ4}, and
\eq{\label{eq:osc estimate}\osc_{\mT_0}= \Big(\sum_{T\in \mT} h_T^2\twonorm{f-\bar f}_T^2+ h_T\twonorm{g-\bar g}_{\pa T\cap \Gamma_1}^2\Big)^\half
         \ls h_0 \big(\twonorm{f}+\twonorm{g}_{\half, \Gamma_1}\big),
}
we have
\eqn{\energy{u-u_n} \ls & \Lcvg^n\big(1+k^2h_n\big) \big(kh_0+k^{-\half}(kh_0)^{\alpha}\big)M(f,g). }
This completes the proof of the theorem.
\end{proof}

\begin{remark}
Note that in order to prove the convergence of AFEM with initial mesh size allowed to be in the preasymptotic regime, we have not, as intuitive, first proved the contraction property of the finite element solutions. Such an approach  requires quasi-orthogonality of the finite element solutions which holds only in the asymptotic regime (see Remark~\ref{qor} and \cite{ACIPforhelmholtz}). But we have first proved the contraction property of the corresponding elliptic projections instead, and then used the equivalence between error estimators for finite element solutions and the elliptic projections and the upper bound to obtain the convergence result.
\end{remark}

\section{Quasi-optimality of AFEM}\label{s4}
In this section, we prove the quasi-optimality of the AFEM, which can be done if we can prove the quasi-optimality of the sequence of elliptic projections $\set{\tildeu{n}}$.

Let $\mathbb{T}$ be the set of all conforming triangulations obtained by refining $\mT_0$ a finite number of times by the newest vertex bisection algorithm.  The full conception of $\mathbb{T}$ can be found in \cite{quasi-optimal}. From Remark~\ref{rUB}(ii),  \eqref{LB1}, and the fact that $\osc_T\le\eta_T$, we know that
\eq{\label{totalerror}\eta_{\mT_n}  \eqsim \big(\hone{u-\tildeu{n}}^2+\osc_{\mT_n}^2\big)^\frac12,}
if $\conditionT \le \Clb$, where the right hand side in \eqref{totalerror} is the so-called \textit{total error} \cite{quasi-optimal} for the elliptic projection $\tildeu{n}$ .
\begin{lemma}[Cea's lemma for the total error of elliptic projection]\label{lem:Cea}
\[ \hone{u-\tildeu{n}}^2 + \osc_{\mT_n}^2(\tildeu{n}) \le  \inf_{V\in \V{n}}\big(\hone{u-V}^2+\osc_{\mT_n}^2(V)\big). \]
\end{lemma}
\begin{proof}
First we note that $\osc_{\mT_n}^2(V)= \sum_{T \in \mT} h_T^2 \twonorm{f-\bar f}_T^2+h_T\twonorm{g-\bar g}_{\pa T\cap \Gamma_1}^2$ for any $V\in\V{n}$, which is independent of
$V$. On the other hand, since $\tildeu{n}$ is  the $H^1$-projection of $u$ onto $\V{n}$ (see \eqref{eq:elliptic fem}), we have
\begin{align*}
\hone{u-\tildeu{n}}\le& \inf_{V\in\V{n}}\hone{u-V}.
\end{align*}
This completes the proof of the lemma.
\end{proof}

\begin{lemma}[Localized upper bound]\label{lem:localize upper bound}
\[\energy{\tildeu{n+1}-\tildeu{n}}^2 \le \Lloc (\eta_{\mathcal{R}_n}^2(u_n)+t^2(k,h_n)\eta_{\mT_n}^2). \]
where $\mathcal{R}_n\subset\mT_n$ is the set of refined elements in refinement step at the $n$-th loop. The constant $\Lloc$ depends only on the shape regularity of $\T{0}.$
\end{lemma}

\begin{proof}
Following the proof of  \cite[Lemma 3.6]{quasi-optimal}, let $\Omega_*=\bigcup \{T:\;T\in\mathcal{R}_n\}$ be the union of refined elements. Let $\Omega_i, i\in\{1,2,...,I\}$ denote the connected components of the interior of $\Omega_*.$ Denote by $\mT_{n,i}=\{T\in \T{n}:T\subset \bar{\Omega}_i\}$ and
 $\V{n,i}:=\{ v \in H^1(\Omega_i):\; v|_{T} \in P_1(T), \forall T \in \mT_{n,i} \}.$ Let $\mI_{n,i}:  H^1(\Omega_i)\rightarrow \V{n,i}$ be the Scott-Zhang interpolation operator over the triangulation $\mT_{n,i}.$

 For the error $E_*:=\tildeu{n+1}-\tildeu{n}\in \V{n+1},$ we construct an approximation $V\in\V{n}$ by
 \begin{align*}
 V:=
 \begin{cases}
 \mI_{n,i}E_* &\text{in} ~ \Omega_i, \\
 E_*       &elsewhere.
 \end{cases}
 \end{align*}
 Since $\mI_{n,i}v=v$ on $\partial \Omega_i$ whenever $v=w$ on $\partial \Omega_i$ for some $w\in \V{n,i},$ we know that $V\in\V{n}.$
 For convenience, we set $\psi=E_*-V$, obviously $\psi=0$ in $\Omega\setminus \Omega_*.$ From Lemma~\ref{lem:SZ},  Lemma~\ref{lem:elliptic projection estimate}, Lemma~\ref{lem:upper bound}, and \eqref{eq:t}, we conclude that
\begin{align*}
b(E_*,E_*)=&b(E_*,E_*-V) = b(u-\tildeu{n},\psi)\\
                        =&b(u-\tildeu{n},\psi-P_h\psi)=b(u-u_n,\psi-P_h\psi)\\
                         =&a(u-u_n,\psi-P_h\psi)+(k^2+1)(u-u_n,\psi-P_h\psi)\\
                          &+\ci k\intbound{u-u_n,\psi-P_h\psi}_{\Gamma_1}\\
                         \lesssim&|a(u-u_n,\psi)|+k^2\twonorm{u-u_n}h_n^\alpha \hone{\psi}
                           +kh_n^\halfa \bdnorm{u-u_n}\hone{\psi}\\
                         \lesssim&\eta_{\mathcal{R}_n}(u_n)\twonorm{\nabla E_*}_{\Omega_*}
                         +t(k,h_n)\eta_{\mT_n}\twonorm{\nabla E_*}_{\Omega_*}.
\end{align*}
Here $P_h$ denotes the elliptic projection defined in \eqref{eq:elliptic fem} but with $\V{\mT}$ replaced by $\V{n}$.
Thus we obtain
\[\hone{\tildeu{n+1}-\tildeu{n}}^2 \lesssim \eta_{\mathcal{R}_n}^2(u_n)+t^2(k,h_n)\eta_{\mT_n}.\]
This completes the proof of the lemma.
\end{proof}

\begin{lemma}[Complexity of refinements]\label{lem:complexity of refinement}
Assume that $\mT_0$ satisfies the conditions of \cite[Lemma 2.3]{quasi-optimal}. Let $\{\mT_n \}_{n \ge 0}$  be any sequence of refinements of $\mT_0$ where $\mT_{n+1}$ is generated from $\mT_n$ by refinement step in the $n$-th loop, then there exist a constant $C$ solely depending on initial triangulation $T_0$ and bisection number $b$ such that
\[ \#\mT_n-\#\mT_0 \le C\sum_{j=0}^{n-1}\#\mM_j,\quad n\ge 1.\]
\end{lemma}
\begin{proof}
From  \cite[Lemma 2.3]{quasi-optimal}, we have
\[\#\mT_n-\#\mT_0 \lesssim \sum_{j=0}^{n-1}\#\tildeM_{j},\]
where the constant hidden in $\lesssim$ solely depend on $T_0$ and $b$.
By the definition \eqref{tildeM} of $\tildeM$, we get $\#\tildeM_j\le 4 \#\mM_j$, and hence the lemma holds.
\end{proof}

In order to analyze the convergence rate of the AFEM,
we define an approximation class $\mathbb{A}_s$ based on the total error. Let  $\mathbb{T}_N \subset \mathbb{T}$ be the set of all possible conforming triangulations generated from $T_0$ with at most $N$ elements more than $\mT_0$:
\[ \mathbb{T}_N:=\{\mT \in \mathbb{T} :\  \#\mT-\#\mT_0\le N\}.\]
We define the  approximation class $\mathbb{A}_s$ to be
\[ \mathbb{A}_s:= \Big\{(v,f,g) :\  \sup_{N>0}N^s\big(\inf_{\mT\in \mathbb{T}_N} \inf_{V\in\VT}
\big(\hone{v-V}^2+\osc_{\mT}^2(v)\big)^{1/2}\big)<\infty \Big\}, \]
if $v\in \mathbb{A}_s$, we define $|v,f,g|_s$ to be
\[ |v,f,g|_s:=\sup_{N>0}N^s\big(\inf_{\mT\in \mathbb{T}_N} \inf_{V\in\VT}\big(\hone{v-V}^2+\osc_{\mT}^2(v)\big)^{1/2}\big).\]

Since the solution $u$ can be split into a regular part and a singular part with singularities around a finite number of points, following the analysis in \cite{Convergence_rate_for_AFEs} we may prove the following lemma.
\begin{lemma}\label{Ahalf} Let $u$ be the exact solution to \eqref{eq:helm}. We have $(u,f,g)\in\mathbb{A}^\frac{1}{2}$ and $|v,f,g|_\frac12\ls kM(f,g).$
\end{lemma}
\begin{proof} Given any positive constant $\de\le (\#\mT_0)^{-\frac12}$,  from \cite[Lemmas 4.4 and 4.9 and Theorems 5.2 and 5.3]{Convergence_rate_for_AFEs} and the decomposition given in Lemma~\ref{thm:regularity_estimates}, there exists a triangulation $\widetilde{\mT}$ which is a refinement of the initial mesh $\mT_0$ and satisfies the following properties:
\eqn{ \norm{u_R-I_{\widetilde{\mT}}u_R}_1&\ls |u_R|_2 \de,\\ \norm{s_j-I_{\widetilde{\mT}}s_j}_1&\ls  \de,\quad j=1,\cdots, J,\\
\#\widetilde{\mT}-\#\mT_0&\ls\de^{-2},\qtq{and} h_T\ls\de\quad\forall T\in\widetilde{\mT},}
where $I_{\widetilde{\mT}}$ is the Lagrange interpolation operator onto $\mathbb{V}_{\widetilde{\mT}}$.
Therefore, from \eqref{eq:decomp},
\eqn{\norm{u-I_{\widetilde{\mT}}u}_1\ls kM(f,g)\de\ls kM(f,g)\big(\#\widetilde{\mT}-\#\mT_0\big)^{-\frac12}. }
On the other hand, similar to \eqref{eq:osc estimate}, we have
\eqn{\osc_{\widetilde{\mT}}(u)\ls \big(\twonorm{f}+\twonorm{g}_{\half, \Gamma_1}\big)h_{\widetilde{\mT}}\ls kM(f,g)\de\ls kM(f,g) \big(\#\widetilde{\mT}-\#\mT_0\big)^{-\frac12}.}
Consequently,
\eqn{ \inf_{\mT\in \mathbb{T}_N} \inf_{V\in\VT}(\hone{u-V}^2+\osc_{\mT}^2(u))^{1/2} \ls  kM(f,g) N^{-\half}}
This completes the proof of the lemma.
\end{proof}

The following lemma says that if the elliptic projections satisfy a suitable total error reduction from $\mT$ to its a refinement $\mT_*$, the error estimators of the coarser finite element solutions must satisfy a D\"{o}rfler property on the set of refined elements in $\mT$. The proof is similar to that of \cite[Lemma 5.9]{quasi-optimal} but contains some differences. We include it here for the reader's convenience.
\begin{lemma}[D\"{o}rfler property]\label{lem:dorfler property}
Assume that the marking parameter $\theta$ satisfies $\theta \in (0, \theta_{*})$, with $\theta_{*}^2:=\frac{1}{2\Llb(1+\Lloc)}<1.$ Let $\mT\in \mathbb{T}$, $\tildeu{} \in \VT$ be the elliptic projection  of $u$, and $u_h\in\V{\mT}$ be the finite solution. Set $\mu:=\frac{1}{2}\big(1-\frac{\theta^2}{\theta_*^2}\big)$, and let $\mT_*$ is any refinement of $\mT$, such that the elliptic projection $\tildeu{*}\in\V{\mT_*}$ of $u$ satisfies
\[\hone{u-\tildeu{*}}^2+\osc_{\mT_{*}}^2(\tildeu{*}) \le \mu\big(\hone{u-\tildeu{}}^2+\osc_{\mT}^2(\tildeu{})\big).\]
Then there exists a constant $\Cmark$ depending only on $\theta^*$ and $\Clb$ in Lemma~\ref{lem:lower bound} such that if $\conditionT\le\Cmark$, then  the set $\mathcal{R}$ of refined elements in $\mT$ satisfies the D\"{o}rfler property
\[ \eta_{\mR}(u_h) \ge \theta \eta_{\mT}(u_h).\]
\end{lemma}

\begin{proof}
By combining the lower bound \eqref{LB1} and the orthogonality in Lemma \ref{lem:orth} and noting that $\osc_T(v)$ is independent of $v\in \V{\mT}$, we deduce under the condition $\conditionT \le \Clb$ that
\begin{align*}
\frac{(1-\mu)}{\Llb}\eta_\mT^2\le&(1-\mu)(\hone{u-\tildeu{}}^2+\osc^2_\mT)\\
\le&\hone{u-\tildeu{}}^2+\osc^2_\mT-(\hone{u-\tildeu{*}}^2+\osc^2_{\mT_{*}})\\
 \le&\hone{\tildeu{*}-\tildeu{}}^2+\osc_{\mR}^2(u_h)\\
 \le&\hone{\tildeu{*}-\tildeu{}}^2+\eta_{\mR}^2(u_h),
\end{align*}
invoking Lemma \ref{lem:localize upper bound}, we have
\begin{align*}
\frac{(1-\mu)}{\Llb}\eta_\mT^2\le
&\Lloc(\eta^2_\mR(u_h)+t^2(k,h_\mT)\eta^2_\mT)+\eta^2_\mR(u_h)\\
 \le&(1+\Lloc)\big(\eta^2_\mR(u_h)+t^2(k,h_\mT)\eta^2_\mT\big).
\end{align*}
By the definitions of $\theta_*$ and $\mu$, we deduce that if $\condition$ is so small that $t(k,h_{\mT})\le\theta^*$, then
\begin{align*}
\eta^2_\mR(u_h)\ge& \Big(\frac{1-\mu}{\Llb(1+\Lloc)}
        -t^2(k,h_\mT)\Big)\eta^2_\mT
        \ge\Big(\frac{1-\mu}{\Llb(1+\Lloc)}
        -\theta^2_*\Big)\eta^2_\mT\\
        \ge&\Big(2\theta^2_*\Big(1-\frac{1}{2}\Big(1-\frac{\theta^2}{\theta^2_*}\Big)\Big)-\theta^2_*\Big)\eta^2_\mT
        =\theta^2\eta^2_\mT.
\end{align*}
This completes the proof of the lemma.
\end{proof}

\begin{lemma}[Cardinality of $\mM_n$]\label{lem:cardinality of Mn}
Assume that the marking parameter $\theta$ satisfies $\theta \in (0, \theta_{*})$, $\condition\le\Cmark.$  Let $ \{ \T{n}, \V{n}, u_n\}_{n \ge 0}$ be the sequence of meshes, finite element spaces, and discrete solutions from the AFEM and let $\tildeu{n}\in \V{n}$ be the elliptic projections of $u$. Then
\eqn{
\#\M{n} \lesssim \Big(1-\frac{\theta^2}{\theta_*^2}\Big)^{-1} |u,f,g|_\frac12^2
\big(\hone{u-\tildeu{n}}^2+\osc_{\T{n}}^2\big)^{-1}.
}
\end{lemma}
\begin{proof}
This lemma can be proved by  mimic the proof of \cite[Lemma 5.10]{quasi-optimal}, just need to replace the Lemmas 5.2 and 5.9 used there with our Lemmas~\ref{lem:Cea} and \ref{lem:dorfler property}, respectively. We omitted the details.
\end{proof}

\begin{theorem}[Quasi-optimality]\label{thm:quasi optimality}
Assume that  $\theta \in (0, \theta_{*})$ where $\theta_{*}$ is defined in Lemma~\ref{lem:dorfler property}. Then there exist a constant $\Copt$ independent of the wave number $k$ and iteration number $n$, such that if $\condition\le\Copt$,
\eq{\label{opt} \energy{u-u_n}\lesssim (1+k^2h_n) kM(f,g) (\#\T{n}-\#\T{0})^{-\half}.}

\end{theorem}

\begin{proof} We first prove the quasi-optimalities of $\tildeu{n}$ and $\tildeeta_n$, i.e.
\eq{\label{opttu} (\hone{u-\tildeu{n}}^2+\gamma\tildeeta_n^2)^\frac{1}{2} \ls C_\ta|u,f,g|_\frac12(\#\T{n}-\#\T{0})^{-\frac12},}
where $C_\ta:=\frac{\Lcvg}{(1-\Lcvg^2)^{\frac12}}\big(1-\frac{\theta^2}{\theta_*^2}\big)^{-\frac{1}{2}}.$ Denote by
$M:=\big(1-\frac{\theta^2}{\theta_*^2}\big)^{-1}|u,f,g|^2_\frac12.$
According to Lemmas \ref{lem:complexity of refinement} and  \ref{lem:cardinality of Mn}, we have
\[
\#\T{n}-\#\T{0}\lesssim\sum_{j=0}^{n-1}\#\M{j}\lesssim M\sum_{j=0}^{n-1}\big(\hone{u-\tildeu{j}}^2+\osc^2_{\T{j}}\big)^{-1}.
\]
Using the \eqref{LB2}, under the condition $\condition \le \Clb$, we deduce  that
\begin{align*}
\hone{u-\tildeu{j}}^2+\gamma\tildeeta_j^2\le&\hone{u-\tildeu{j}}^2
                                         + \gamma\tLlb (\hone{u-\tildeu{j}}^2+\osc^2_{\T{j}})\\
                                         \le&(1+\gamma  \tLlb )(\hone{u-\tildeu{j}}^2+\osc^2_{\T{j}}).
\end{align*}
Meanwhile \eqref{contraction} in the proof of Theorem \ref{thm:contraction property} implies directly that, if $\condition\le \Ccvg$,
\[\hone{u-\tildeu{n}}^2+\gamma\tildeeta_n^2\le \Lcvg^{2(n-j)}\big(\hone{u-\tildeu{j}}^2+\gamma\tildeeta_j^2\big), \quad  j=0,1,..., n-1. \]
Thus
\begin{align*}
\big(\hone{u-\tildeu{j}}^2+\osc^2_{\T{j}}\big)^{-1}\le&(1+\gamma\tLlb)\big(\hone{u-\tildeu{j}}^2+\gamma\tildeeta_j^2\big)^{-1}\\
                     \le&(1+\gamma\tLlb)\Lcvg^{2(n-j)}\big(\hone{u-\tildeu{n}}^2+\gamma\tildeeta_n^2\big)^{-1}.
\end{align*}
This implies that
\begin{align*}
\#\T{n}-\#\T{0}\lesssim& M(1+\gamma\tLlb)\big(\hone{u-\tildeu{n}}^2+\gamma\tildeeta_n^2\big)^{-1}\sum_{j=1}^n\Lcvg^{2j}\\
                     \lesssim&\frac{\Lcvg^2}{1-\Lcvg^2}M(1+\gamma\tLlb)\big(\hone{u-\tildeu{n}}^2+\gamma\tildeeta_n^2\big)^{-1}.
\end{align*}
Therefore
\[ (\hone{u-\tildeu{n}}^2+\gamma\tildeeta_n^2)^\frac{1}{2}\lesssim \frac{\Lcvg}{\big(1-\Lcvg^2\big)^\frac12}M^\frac12(1+\gamma\tLlb)^\frac{1}{2}(\#\T{n}-\#\T{0})^{-\frac12},\]
that is, \eqref{opttu} holds.

Then \eqref{opt} follows from the equivalence $\eta_n\eqsim\tildeeta_n$ given in Corollary~\ref{cor:equivalent of eta} and the upper bound \eqref{UB1}. This completes the proof of the theorem.
\end{proof}

\begin{remark} Given $\frac12<s<\al$, under the condition that $k^{1+s}h_0^s$ is sufficiently small, the convergence of the the adaptive CIP-FEM including AFEM are proved in \cite{ACIPforhelmholtz}. Our results in Theorems~\ref{thm:contraction property} and \ref{thm:quasi optimality} for AFEM hold for larger $h_0$, in fact, in the preasymptotic regime.
\end{remark}

\section{Numerical tests}\label{s:num} In this section, we present some numerical results to verify our theoretical findings and the performance of AFEM and adaptive CIP-FEM.
\subsection{CIP-FEM}\label{ss:cipfem} We first recall the linear CIP-FEM \cite{dd76,Du2015Preasymptotic,wu2014,zw2013},  that is done
by adding some appropriate
penalty terms on the jumps of the fluxes across interior edges to the finite element system \eqref{eq:femhelm} while using the same approximation space as it.

We define the ``energy" space $V$ and the sesquilinear
form $a^\ga(\cdot,\cdot)$ on $V\times V$ as
\begin{align}
  V&:=H^1(\Om)\cap\prod_{K\in\mT_h} H^2(K), \nonumber \\
  \label{eah}
  a^\ga(u,v)&:=a(u, v)+  J(u,v)\qquad\forall\, u, v\in V,\\
  J(u,v)&:=\sum_{e\in\mE^I}\ga_e\, h_e \pd{\jump{\na u},\jump{\na v}},
          \label{eJ}
\end{align}
where  $\ga_{e}$ for $e\in\mE^I$ are called the penalty parameters,  which are
complex numbers with nonpositive imaginary parts due to $-\ci k$ in the boundary condition on $\Gamma_1$.
It is clear that $J(u,v)=0$ if $u\in H^2(\Om)$ and $v\in V$. Therefore, if $u\in H^2(\Om)$ is the solution to \eqref{eq:helm}, then
\begin{equation*}
  a^\ga(u,v) =(f,v)+\pd{g, v},\quad\forall v\in V.
\end{equation*}
This motivates the definition of the CIP-FEM: Find $u_h^\ga\in \VT$ such that

\begin{equation}\label{ecipfem}
  a^\ga(u_h^\ga,v_h) =(f, v_h)+\pd{g, v_h},  \;\forall v_h\in \VT.
\end{equation}

Compared with our earlier standard FEM \eqref{eq:femhelm}, the CIP-FEM \eqref{ecipfem} has added a bilinear form $J(u,v)$ that collects the so-called penalty terms, one from each interior edge of $\mT$. Clearly, the CIP-FEM reduces to the standard FEM \eqref{eq:femhelm} when the penalty parameters $\ga_e$ in $J(u,v)$ are turned off.

Recent theoretical results and numerical evidences show that  the penalty parameters may be tuned to reduce the pollution errors significantly (see \cite{liwu2019,zbw16,Du2015Preasymptotic,wu2014,zw2013}). By following the technical derivations in Sections~\ref{s3}--\ref{s4}, we can establish the convergence and quasi-optimality of the adaptive CIP-FEM in Theorem~\ref{thm:contraction property} and Theorem~\ref{thm:quasi optimality} also for the adaptive CIP-FEM. We omit the tedious technical details here.

\subsection{Numerical example}\label{ss:example1} We consider the Helmholtz equation \eqref{eq:helm} with  $\Omega_1 = (-0.5, 0.8)\times (-0.5, 0.5)$,  $\Omega_0$ a drop-shaped domain which merges a triangle and a circle centered at $(0.5, 0)$. The apical point stands at $(0, 0)$ with angle $\pi /15$, and the two edges are symmetric about the $x$-axis and both tangent to the circle. The source terms $f$ and $g$ are so chosen that the exact solution $u=\phi(r) J_\frac{15}{29}(kr)\sin(\frac{15}{29}(\theta-\frac{\pi}{30}))$, where $ J_\alpha(r)$ with $r=(x^2+y^2)^\half $ stands for the Bessel function of the first kind, the cutoff function $\phi(r)$ is defined below with $R=0.5 \cos(\frac{\pi}{30})$.
\eqn{
\phi(r)=
\begin{cases}
(1+\frac{2r}{R})(1-\frac{r}{R})^2,&  \ r\le R,\\
0,&       \ x>R.
\end{cases}
}
Obviously $u=0$ on $\Gamma_0.$
The penalty parameters in the CIP-FEM are chosen as $\ga_e\equiv\ga:=-\frac{\sqrt{3}}{24}-0.005 \ci$ where the real part of $\ga$ can remove leading term of the dispersion error on equilateral triangulations \cite{hw12} and the imaginary part can enhance the stability of the adaptive CIP-FEM.  The codes are written in MATLAB. We use the program ``initmesh.m" to generate the initial meshes in which most triangle elements are approximate equilateral triangles. Since the local refinements may reduce the quality of the meshes, resulting a decrease of the effect of the selected penalty parameter, we use ``jigglemesh.m" with default arguments to improve the mesh quality in each iteration of adaptive CIP-FEM. We set the initial mesh size $h_0\approx \frac{\pi}{k}$ which is about half the wavelength and is necessary for the FE interpolant (or the elliptic projection) to resolve the wave. Although the theoretical results of convergence and quasi-optimality in Theorems~\ref{thm:quasi optimality} requires that the initial mesh size satisfies that $k^3h_0^{1+\al}$ is sufficiently small, our numerical tests show that our adaptive algorithms still work for coarser initial meshes. This gap between theory and practise still needs to be filled.   Moreover, the error estimator is rescaled by a factor of 0.15 as in MATLAB PDE toolbox.

We present in Figure \ref{fig: mesh and exact solution} a sample mesh and the magnitude of the CIP-FE solution in the adaptive CIP-FEM for $k=15\pi $ after 16 local refinements. The mesh contains $79909$ triangles and the error relative to the exact solution is $2.64\%$ at this step. It can be seen that the mesh has the same wave pattern as the numerical solution, and is denser around the corner where the singularity  is located.
\par
\begin{figure}[ht]
\includegraphics[height=1.05in,width=4.5in]{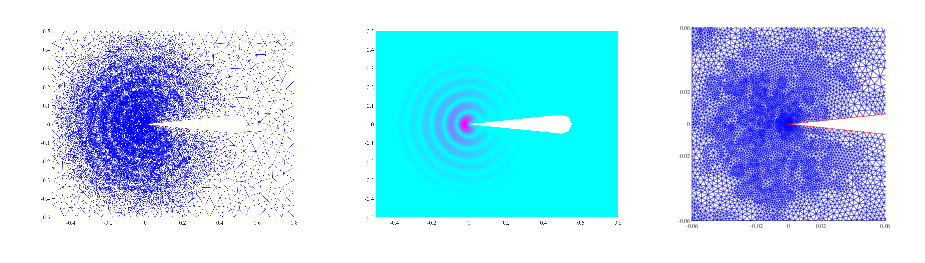}
\caption{A sample mesh (left) and the magnitude (middle) of the CIP-FE solution for $k=15\pi$ after 16 local refinements, the right figure is the local view of the mesh around $(0, 0)$. }
\label{fig: mesh and exact solution}
\end{figure}

\begin{figure}[ht]
\includegraphics[height=2in,width=5in]{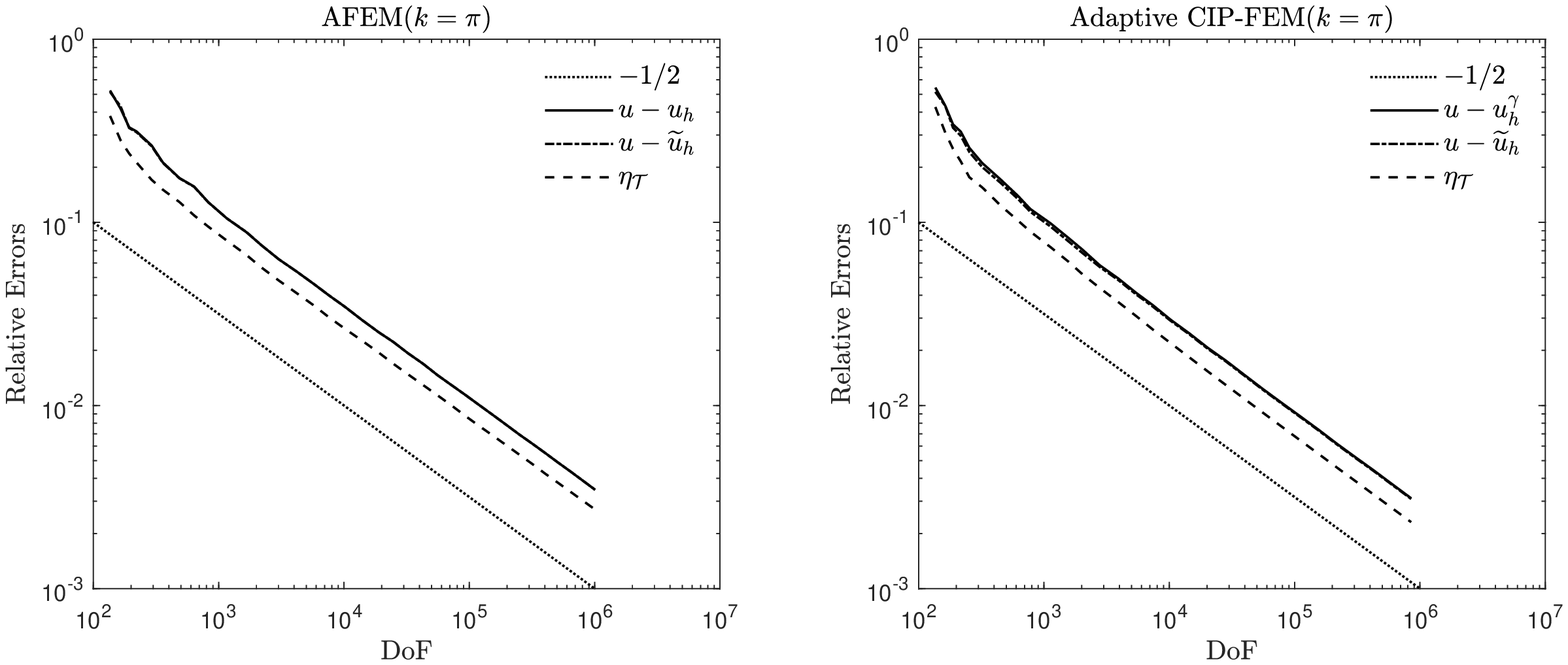}
\caption{$k=\pi$: Relative errors of the numerical solution (solid), the elliptic projection (dash-dotted), the error estimator (dashed) versus the number of degrees of freedom for AFEM (left) and adaptive CIP-FEM (right). The dotted lines give the reference slope of $-\frac12$.}
\label{fig: k_5pi}
\end{figure}

For a small wave number $k=\pi $, Figure~\ref{fig: k_5pi} plots relative errors of the numerical solution, the elliptic projection, the error estimator versus the number of degrees of freedom (DoF) in  the logarithm scale for AFEM (left) and adaptive CIP-FEM (right), respectively. For the AFEM, it is observed from Figure~\ref{fig: k_5pi}(left) that the error estimator $\eta_\mT$ is a good estimate of both the error $\energy{u-u_h}$ of the FE solution  and the error $\energy{u-\tildeu{h}}$ of the elliptic projection. All of them decay at the optimal rate of $O(N^{-\half)}$. Figure~\ref{fig: k_5pi}(right) shows similar behaviour for the adaptive CIP-FEM, except the error $\energy{u-u_h^\ga}$ of the CIP-FE solution behaves a little better than that of the FE solution at the beginning of the adaptive iterations.

\begin{figure}[ht]
\includegraphics[height=2in,width=5in]{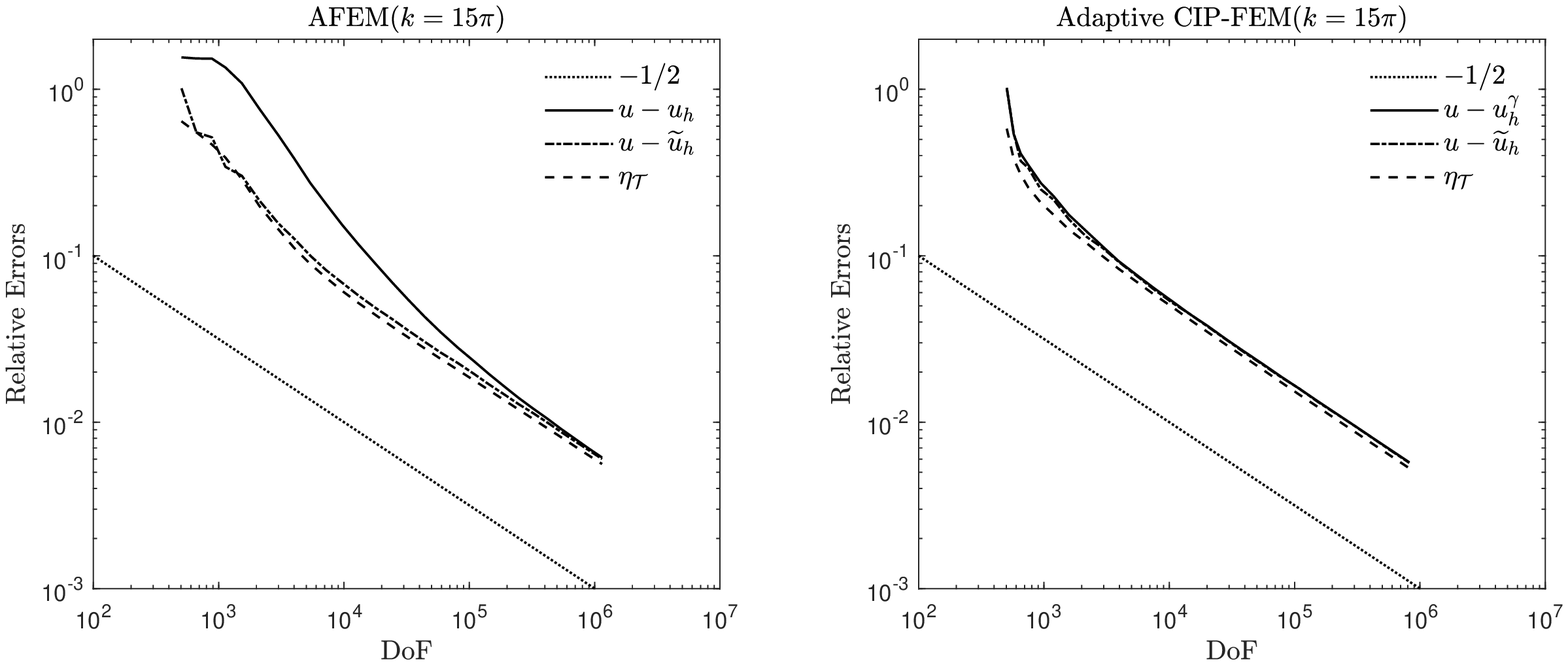}
\caption{$k=15\pi$: Relative errors of the numerical solution (solid), the elliptic projection (dash-dotted), the error estimator (dashed) versus the number of degrees of freedom for AFEM (left) and adaptive CIP-FEM (right). The dotted lines give the reference slope of $-\frac12$. }
\label{fig: k_15pi}
\end{figure}

For a medium wave number $k=15\pi$,  the corresponding numerical results are plotted in Figure~\ref{fig: k_15pi}. It is observed from Figure~\ref{fig: k_15pi}(left) for the AFEM that the error estimator $\eta_\mT$ fits the error $\energy{u-\tildeu{h}}$ of the elliptic projection well and both of them decay at the optimal rate of $O(N^{-\half })$. However, due to the pollution effect, the error $\energy{u-u_h}$ of the FE solution doesn't decrease at the beginning, but decreases rapidly after a few iterations, and then  decays at the optimal rate after the pollution disappears. The  error estimator obviously underestimates the true error of the finite element solution in the preasymptotic regime.
Figure~\ref{fig: k_15pi}(right) demonstrates that CIP-FEM can efficiently reduce the pollution error, and as a result, the error estimator does not underestimate the error of the CIP-FE solution even in the preasymptotic regime.

\begin{figure}[ht]
\includegraphics[height=2in,width=5in]{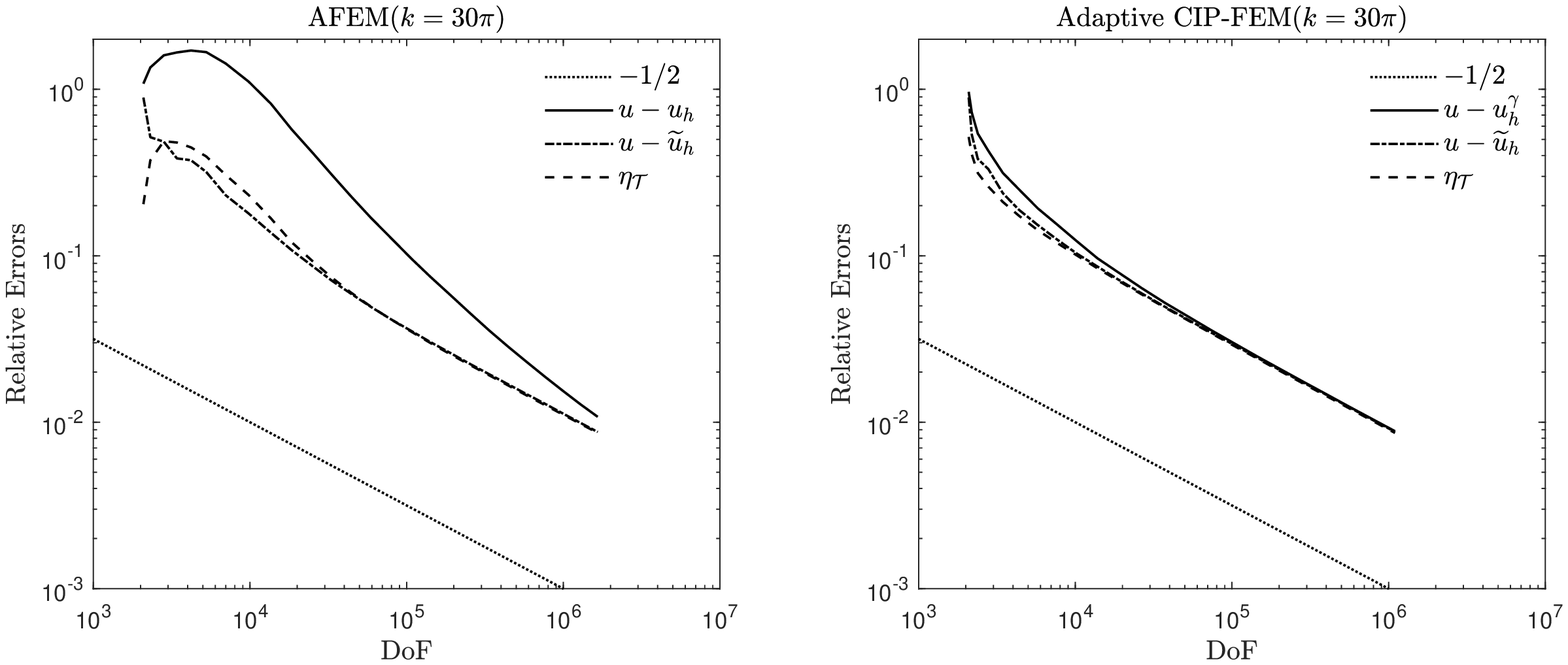}
\caption{$k=30\pi$: Relative errors of the numerical solution (solid), the elliptic projection (dash-dotted), the error estimator (dashed) versus the number of degrees of freedom for AFEM (left) and adaptive CIP-FEM (right). The dotted lines give the reference slope of $-\frac12$.}
\label{fig: k_30pi}
\end{figure}

For a larger wave number $k=30 \pi $, the corresponding numerical results are plotted in Figure~\ref{fig: k_30pi}. The pollution effect becomes worse than $k=15\pi$ and the error estimator $\eta_\mT$ underestimates the error of the FE solution more seriously. Meanwhile, the adaptive CIP-FEM can still greatly reduce the pollution error and performs well.

\begin{figure}[ht]
\includegraphics[height=2in,width=5in]{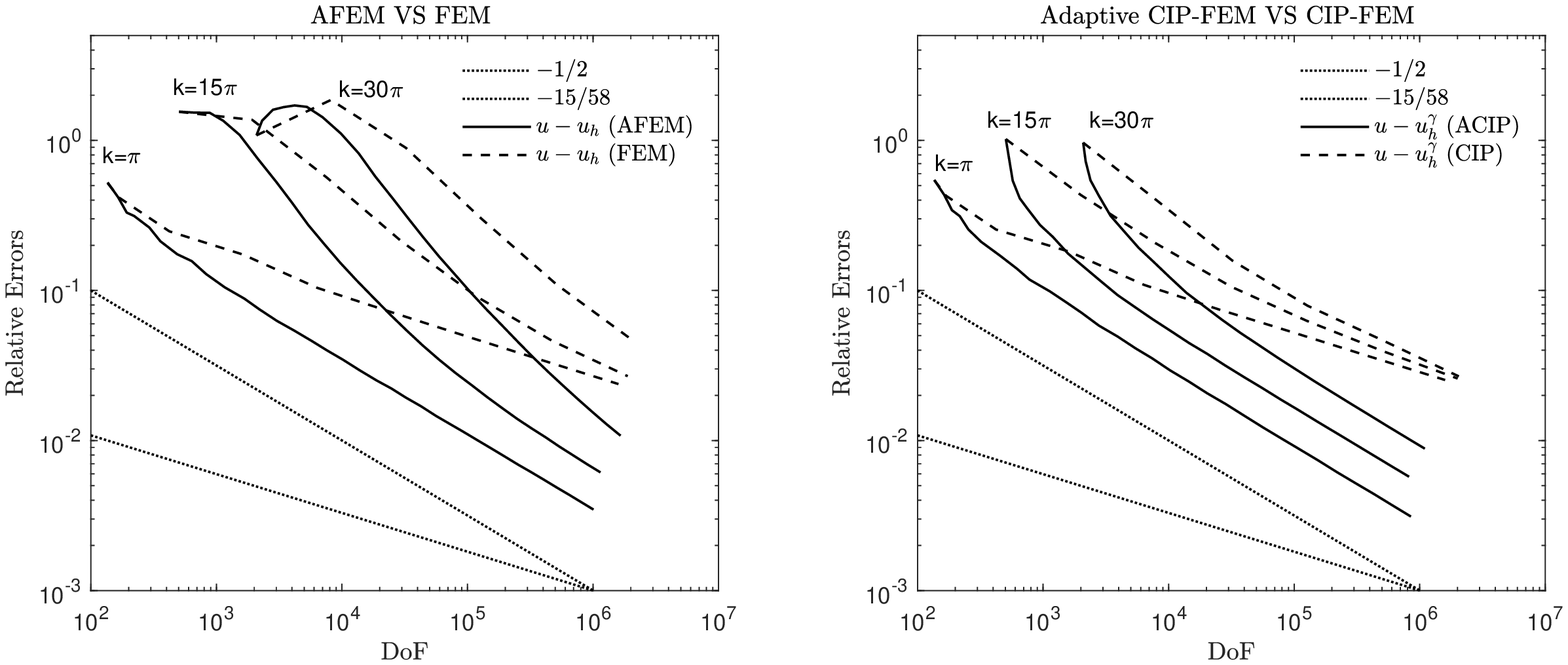}
\caption{Left: Relative errors  of AFEM (solid) and FEM (dashed) versus the number of degrees of freedom. Right: Corresponding plots for adaptive CIP-FEM and CIP-FEM. The dotted lines give the reference slopes of $-\frac12$ and $-\frac{15}{58}$, respectively.}
\label{fig: AFEM_FEM_ACIP_CIP}
\end{figure}

Finally, we give a comparison in Figure~\ref{fig: AFEM_FEM_ACIP_CIP} between two adaptive methods (i.e., AFEM and adaptive CIP-FEM) and two corresponding methods using uniform refinements (called FEM and CIP-FEM for simplicity).
For $k=\pi$, we can see from Figure~\ref{fig: AFEM_FEM_ACIP_CIP}(left) the error $\energy{u - u_h}$ of the numerical solution of AFEM decays at the optimal rate of $O(N^{-\half})$, while for FEM, after a few iterations it decays at the rate of $O(N^{-\frac{15}{58}})$, since the interpolation error due to singularity at the original begins to dominate. For $k=15\pi$ and $k=30\pi$, due to the pollution effect, the energy errors of the solutions of the AFEM and FEM  do not decrease at the beginning, but decrease rapidly after a few iterations, and then the error of AFEM decays at the the optimal rate of $O(N^{-\half})$ while the error of the FEM tends to decay at the rate of $O(N^{-\frac{15}{58}})$. Clearly, the adaptive FEM performs better than the FEM using uniform refinements. Figure~\ref{fig: AFEM_FEM_ACIP_CIP}(right) shows similar phenomena for the adaptive CIP-FEM and CIP-FEM, except for the lack of the beginning stage where the error does not decrease. Again, the adaptive CIP-FEM performs much better than the CIP-FEM using uniform refinements for this Helmholtz problem with conner singularity, even in the case of high frequency.

\providecommand{\bysame}{\leavevmode\hbox to3em{\hrulefill}\thinspace}
\providecommand{\MR}{\relax\ifhmode\unskip\space\fi MR }
\providecommand{\MRhref}[2]{%
  \href{http://www.ams.org/mathscinet-getitem?mr=#1}{#2}
}
\providecommand{\href}[2]{#2}

\end{document}